\documentclass[11pt,           % in 11pt fonts
final,                         % start in draft mode
english                       % default language
]{article}

\usepackage[english,german,strings]{babel}     % with explicit language
\usepackage{amsmath}           % ams mathematical stuff
\usepackage[utf8]{inputenc}    % smart input of funny chars
\usepackage[T1]{fontenc}       % also for the font encoding
\usepackage{longtable}         % tables longer than one page
\usepackage{exscale}           % large summation signs in 11pt
\usepackage[sort]{cite}        % nicer citations
\usepackage{array}             % nice tables
\usepackage{fancyhdr}          % nicer headers
\usepackage[a4paper]{geometry} % geometry of page layout
\usepackage{xspace}            % better spacing after macros
\usepackage{tikz}              % for commutative diagrams and stuff
\usepackage{nchairx} % additional packages
\usepackage[expansion=false    % no font expansion
           ]{microtype}        % only protrusion
\usepackage[nottoc]{tocbibind} % refs and index in the toc
\usepackage[
final=true,                    % treat as final
plainpages=false,              % whatever that means
pdfpagelabels=true,            % strange options
pdfencoding=auto,              % getting worse
unicode=true,                  % unicode is always nice
hypertexnames=true,            % needs both to get index *and*
naturalnames=true              % crossrefs correct
]{hyperref}                    % hyperrefs are cool!

\usetikzlibrary{arrows,cd}

\newcommand{\myname}{\textbf{Stefan Waldmann}}
\newcommand{\myemail}{\texttt{stefan.waldmann@mathematik.uni-wuerzburg.de}}
\newcommand{\myaddress}{Julius Maximilian University of Würzburg \\
     Department of Mathematics \\
     Chair of Mathematics X (Mathematical Physics) \\
     Emil-Fischer-Straße 31 \\
     97074 Würzburg \\
     Germany}

\newcommand{\AuthorEmailOne}{\texttt{{marvin.dippell@mathematik.uni-wuerzburg.de}}}
\newcommand{\AuthorEmailTwo}{\texttt{felix.menke@mathematik.uni-wuerzburg.de}}
\newcommand{\AuthorEmailThree}{\texttt{\myemail}}

\newcommand{\AuthorOne}{\textbf{Marvin Dippell}}
\newcommand{\AuthorTwo}{\textbf{Felix Menke}}
\newcommand{\AuthorThree}{\textbf{\myname}}

\author{\AuthorOne\thanks{\AuthorEmailOne},
        \addtocounter{footnote}{2}
        \AuthorTwo\thanks{\AuthorEmailTwo},
        \textbf{and}
        \AuthorThree\thanks{\AuthorEmailThree}\\[0.2cm]
        \AuthorAddressThree
        \\[0.5cm]
}

\newcommand{\AuthorAddressThree}{\myaddress}

\newcommand{\Transport}{\mathrm{T}}
\newcommand{\Proj}{\categoryname{Proj}}
\newcommand{\VectTriple}{\categoryname{Vect_3}}
\newcommand{\VBQuad}[1]{\underline{#1}}

\newcommand{\cCoisoModTriple}[1]{\categoryname{C_3Mod(}#1\categoryname{)}}
\newcommand{\cModules}[1]{\categoryname{Mod(}#1\categoryname{)}}
\newcommand{\CoisoModTriple}{\categoryname{C_3Mod}}
\newcommand{\CoisoHom}{\categoryname{C_3}\!\operatorname{Hom}}
\newcommand{\coisoField}[1]{{\underline{\field{#1}}}}
\newcommand{\CoisoAlgTriple}{\categoryname{C_3Alg}}
\newcommand{\CoisoSetTriple}{\categoryname{C_3Set}}
\newcommand{\CoisoSetTripleInj}{\categoryname{C_3Set^\mathrm{inj}}}
\newcommand{\SetTriple}{\categoryname{Set_3}}
\newcommand{\SetTripleInj}{\categoryname{Set_3^\mathrm{inj}}}
\newcommand{\Total}{\mathrm{tot}}
\newcommand{\Wobs}{\mathrm{N}}
\newcommand{\Null}{\mathrm{0}}

\title{A Serre-Swan Theorem for Coisotropic Algebras}

\date{December 2020}

\begin{document}

\selectlanguage{english}

%
% title page
%

\maketitle

%
% abstract
%

\begin{abstract}
    Coisotropic algebras are used to formalize coisotropic reduction
    in Poisson geometry as well as in deformation quantization and
    find applications in various other fields as well. In this paper
    we prove a Serre-Swan Theorem relating the regular projective
    modules over the coisotropic algebra built out of a manifold $M$,
    a submanifold $C$ and an integrable smooth distribution
    $D \subseteq TC$ with vector bundles over this geometric
    situation and show an equivalence of categories for the case of
    a simple distribution.
\end{abstract}

%
% table of contents
%

\tableofcontents
\newpage

%
% Introduction
%

\section{Introduction}
\label{sec:Introduction}

In Poisson geometry, coisotropic reduction is one of the common
reduction schemes to construct out of a Poisson manifold $(M, \pi)$
with a coisotropic submanifold $\iota\colon C \to M$ a new Poisson
manifold: being coisotropic means that the Hamiltonian vector fields
$X_f \in \Secinfty(TM)$ are tangent to $C$ for all functions
$f \in \Cinfty(M)$ with $\iota^*f = 0$. Thus they define a smooth, but
in general singular, distribution which turns out to be integrable
nevertheless. The quotient of $C$ with respect to the orbit relation
$\sim$ induced by this distribution then becomes a Poisson manifold
$M_\red = C \big/ \mathord{\sim}$ with Poisson structure $\pi_\red$
whenever it is actually a manifold at all. This construction
generalizes the Marsden-Weinstein reduction at momentum level zero,
where the coisotropic submanifold is the level set of the momentum map
for value zero.

Being omnipresent in Poisson geometry it raises immediately the
question whether and how one can pass to a quantum analogue. In fact,
this question was first posed by Dirac \cite{dirac:1964a} when
discussing the quantization of first class constraints which, in
modern language, are essentially coisotropic submanifolds. The physical
necessity of understanding such constraint systems comes from quantum
field theoretical models whenever gauge degrees of freedom are
involved.

Within deformation quantization \cite{bayen.et.al:1978a} many general
results on the quantization of coisotropic submanifolds have been
obtained, most notably by Bordemann \cite{bordemann:2005a} for the
global aspects in the symplectic case and by Cattaneo and Felder
\cite{cattaneo.felder:2004a, cattaneo.felder:2007a} for the general
Poisson case. The principal idea is to find a star product $\star$ on
$(M, \pi)$ in such a way that the classical vanishing ideal
$\algebra{J}_C$ of the constraint surface $C$ becomes or is deformed
into a left ideal inside $\Cinfty(M)\formal{\lambda}$ with respect to
$\star$. Moreover, one wants the classical Lie normalizer of
$\algebra{J}_C$ with respect to the Poisson bracket to be deformed
into the Lie normalizer with respect to the star product
commutator. Then the quotient algebra of this normalizer by the
(therein two-sided associative) ideal $\algebra{J}_C$ should provide a
quantization of the reduced phase space $(M_\red, \pi_\red)$. In
simple enough geometric situations this program is shown to be
successful in the above references.

Slightly more general and beyond the original Poisson geometric
motivation is the situation of a manifold $M$ with a submanifold $C$
being equipped with a possibly singular but integrable smooth
distribution $D \subseteq TC$. This will be our geometric framework in
this paper. In particular, in this case one can still consider the orbit
space $M_\red = C / D$ whenever it turns out to be a manifold again.

In \cite{dippell.esposito.waldmann:2019a,
  dippell.esposito.waldmann:2020a} a more algebraic and conceptual
approach was proposed: the notion of coisotropic (triples of) algebras
is designed to provide an algebraic formulation for the above
reduction scheme. One considers a total algebra $\algebra{A}_\Total$
together with another algebra $\algebra{A}_\Wobs$ and an algebra
homomorphism $\iota\colon \algebra{A}_\Wobs \to
\algebra{A}_\Total$. Finally, we require a two-sided ideal
$\algebra{A}_\Null \subseteq \algebra{A}_\Wobs$. Then such triples
$\algebra{A} = (\algebra{A}_\Total, \algebra{A}_\Wobs,
\algebra{A}_\Null)$ of algebras are suitable to encode both the
classical situation as well as the quantized version. They always
admit a reduction given by the quotient
$\algebra{A}_\red = \algebra{A}_\Wobs / \algebra{A}_\Null$. In the
previous works a good category of bimodules over such triples of
algebras was studied in detail allowing for Morita theory.  In
addition, their deformation theory was discussed, including a first
investigation of the relevant Hochschild cohomologies.

In our geometric setting, we can associate a coisotropic algebra
$\Cinfty(M, C, D)$ by taking $\Cinfty(M)$ as $\Total$-component, the
functions on $M$ whose restrictions to $C$ are constant in direction
of $D$ as $\Wobs$-component and the vanishing ideal of $C$ as
$\Null$-component. If the quotient $C / D$ is a manifold, the reduced
algebra $\Cinfty(M, C, D)_\red$ becomes isomorphic to
$\Cinfty(M_\red)$.

While \cite{dippell.esposito.waldmann:2019a,
  dippell.esposito.waldmann:2020a} focus on general algebraic features,
the aim of this paper is to give a first contact to the underlying
geometry for a particular class of coisotropic algebras by formulating
a Serre-Swan Theorem in this context.

Recall that the classical Serre-Swan Theorem for smooth manifolds
identifies vector bundles over manifolds with finitely generated
projective modules over the algebra of smooth functions.  Thus, we
would like to understand vector bundles over manifolds that are
equipped with a distribution on a submanifold by comparing them to
certain projective modules over the corresponding coisotropic algebra
$\Cinfty(M, C, D)$.

To achieve such a theorem we first have to specify what
\emph{projective} module over a coisotropic algebra should really
mean: this is complicated by the fact that the module category of a
coisotropic algebra is \emph{not} an abelian category. Hence the usual
standard definition has to be interpreted in the correct way. We base
our definition on a notion of free modules and a splitting
property with respect to \emph{regular} epimorphisms instead of all
epimorphisms. This seems to be the most reasonable choice for a
category $\Proj(\algebra{A})$ of projective modules over a coisotropic
algebra $\algebra{A}$.

On the geometric side, we investigate vector bundles $E_\Total \to M$
with a specified subbundle $E_\Wobs \to M$ and a subbundle
$E_\Null \subseteq \iota^\sharp E_\Wobs \to C$ inside the pull-back
bundle on the submanifold $C$. Moreover, we need a holonomy-free
$D$-connection for the vector bundle $\iota^\sharp E_\Wobs$ preserving
the subbundle $E_\Null$. With an obvious notion of morphisms such
$(E_\Total, E_\Wobs, E_\Null, \nabla)$ form a category
$\VectTriple(M, C, D)$.

Taking sections of such $(E_\Total, E_\Wobs, E_\Null, \nabla)$ then
yields a module over the corresponding coisotropic algebra
$\Cinfty(M, C, D)$.  However, without additional assumptions it will
be hard to control whether this module is indeed projective.

Conversely, if we have a finitely generated regular projective
module over $\Cinfty(M, C, D)$, then we can construct
$(E_\Total, E_\Wobs, E_\Null, \nabla)$ whose sections reproduce the
module we started with: here we do not yet need any additional
assumptions about $D$.

The final Serre-Swan Theorem we state assumes that the quotient
$M_\red = C/D$ is a manifold, i.e. $D$ is a simple distribution. Then
the functor of taking sections yields the equivalence of categories
\begin{equation}
    \label{eq:SerreSwan}
    \Secinfty\colon
    \VectTriple(M, C, D)
    \to
    \Proj(\Cinfty(M, C, D)).
\end{equation}
We finally show that reduction of modules matches the geometric
reduction of vector bundles in a functorial way.

The Serre-Swan Theorem raises several questions some of which will be
pursued in future works:
\begin{itemize}
\item Having a meaningful definition of projective modules one can
    start investigating the algebraic $K$-theory of coisotropic
    algebras, both in the commutative case but also in general. The
    Serre-Swan Theorem then provides a passage to geometric $K$-theory
    adapted to submanifolds equipped with an integrable
    distribution. In this context it would be interesting to
    understand the deformation quantization of projective modules, see
    also \cite{bursztyn.waldmann:2000b}.
\item It would be interesting to understand in which scenarios of
    non-simple distributions the equivalence \eqref{eq:SerreSwan}
    still holds and when it ultimately breaks down.
\item Coisotropic algebras arise in other contexts in differential
    geometry as well: one particularly interesting situation would be
    a submanifold $C \subseteq M$ in an ambient manifold with a group
    action of a discrete group $G$ on $C$. Then the $\Wobs$-component
    would be smooth functions on $M$ whose restrictions to $C$ are
    $G$-invariant and the $\Null$-component is still the vanishing
    ideal.  In this situation, one can also try to find an analogue of
    the above Serre-Swan Theorem.
\end{itemize}

%
% Acknowledgements
%

\noindent
\textbf{Acknowledgements:} It is a pleasure to thank Chiara Esposito
for various stimulating discussions.

%
% Coisotropic Structures
%

\section{Coisotropic Structures}
\label{sec:CoisotropicStructures}

In order to formulate a coisotropic version of the Serre-Swan Theorem
we first need to find a well-behaved notion of projective coisotropic
module.  To this end we will first recall the definitions and some
properties of coisotropic modules and algebras, before studying
coisotropic index sets and the free coisotropic modules generated by
such sets.  This will finally allow us to find a reasonable definition
of projective coisotropic modules and prove many characterizations
analogous to the classical case, such as a dual basis lemma and the
equivalence to direct summands of free modules.  In the following,
$\field{k}$ denotes a commutative unital ring.

%
% Coisotropic Algebras and their Modules
%

\subsection{Coisotropic Algebras and their Modules}
\label{sec:CoisotropicAlgebrasAndModules}

In this preliminary section we introduce coisotropic algebras and
their modules following \cite{dippell.esposito.waldmann:2019a,
  dippell.esposito.waldmann:2020a}.  For this we will first need to
consider coisotropic $\field{k}$-modules as the fundamental algebraic
structure underlying coisotropic algebras and their modules.
\begin{definition}[Coisotropic $\field{k}$-modules]
    \label{definition:CoisoTriplesModules}%
    \begin{definitionlist}
    \item \label{item:CoisoTripleModule} A triple
        $\module{E} = (\module{E}_\Total, \module{E}_\Wobs,
        \module{E}_\Null)$ of $\field{k}$-modules together with a
        module homomorphism
        $\iota_\module{E} \colon \module{E}_\Wobs \longrightarrow
        \module{E}_\Total$ is called a \emph{coisotropic
          $\field{k}$-module} if
        $\module{E}_\Null \subseteq \module{E}_\Wobs$ is a sub-module.
    \item \label{item:CoisoTripleModuleMorphism} A morphism
        $\Phi\colon \module{E} \longrightarrow \module{F}$ between
        coisotropic $\field{k}$-modules is a pair
        $(\Phi_\Total, \Phi_\Wobs)$ of module homomorphisms
        $\Phi_\Total \colon \module{E}_\Total \longrightarrow
        \module{F}_\Total$ and
        $\Phi_\Wobs\colon \module{E}_\Wobs \longrightarrow
        \module{F}_\Wobs$ such that
        $\Phi_\Total \circ \iota_\module{E} = \iota_{\module{F}} \circ
        \Phi_\Wobs$ and
        $\Phi_\Wobs(\module{E}_\Null) \subseteq \module{F}_\Null$.
    \item The category of coisotropic $\field{k}$-modules is denoted
        by $\CoisoModTriple$.
    \end{definitionlist}
\end{definition}
Note that we suppress the ring $\field{k}$ in our notation.

We now collect some categorical properties of coisotropic modules.
All of the following can be proved by directly checking the
categorical properties, see e.g. \cite{maclane:1998a}.  For this let
$\module{E}$, $\module{F}$ be coisotropic modules and let
$\Phi, \Psi \colon \module{E} \to \module{F}$ be morphisms of
coisotropic modules.
\begin{remarklist}[label=\alph*)]
\item The morphism $\Phi$ is a monomorphism iff $\Phi_\Total$ and
    $\Phi_\Wobs$ are injective module homomorphisms.
\item The morphism $\Phi$ is an epimorphism iff $\Phi_\Total$ and
    $\Phi_\Wobs$ are surjective module homomorphisms.
\item The morphism $\Phi$ is a regular monomorphism iff it is a
    monomorphism with
    $\Phi_\Wobs^{-1}(\module{F}_\Null) = \module{E}_\Null$.
\item The morphism $\Phi$ is a regular epimorphism iff it is an
    epimorphism with $\Phi_\Wobs(\module{E}_\Null) = \module{F}_\Null$.
\end{remarklist}

Note that monomorphisms (epimorphisms) in
$\CoisoModTriple$ do in general not agree with
regular monomorphisms (epimorphisms), showing that
$\CoisoModTriple$ is not an abelian category, unlike
the usual categories of modules.
\begin{remarklist}[resume,label=\alph*)]
\item \label{Kernel} The kernel of $\Phi$ is given by the coisotropic
    module
    \begin{equation}
	\ker(\Phi)
        =
	\bigl(\ker(\Phi_\Total),
	\ker(\Phi_\Wobs),
	\ker(\Phi_\Wobs) \cap \module{E}_0\bigr)
    \end{equation}
    and $\iota_{\ker} \colon \ker(\Phi_\Wobs) \to \ker(\Phi_\Total)$
    the morphism induced by $\iota_{\module{E}}$.
\item \label{CoisoImage} The image of $\Phi$ is given by the
    coisotropic module
    \begin{equation}
	\image(\Phi)
        =
        \bigl(
            \image(\Phi_\Total),
            \image(\Phi_\Wobs),
            \image(\Phi_\Wobs\at{\module{E}_\Null})
        \bigr)
    \end{equation}
    and
    $\iota_{\image} \colon \image(\Phi_\Wobs) \to \image(\Phi_\Total)$
    the morphism induced by $\iota_{\module{F}}$.
\item \label{InternalHom} The morphisms $\Hom(\module{E},\module{F})$
    of coisotropic modules over $\field{k}$ form a $\field{k}$-module
    which can be enlarged to a coisotropic $\field{k}$-module by
    defining
    \begin{equation}
        \label{Eq:Hom_Module}
	\begin{split}
            \CoisoHom(\module{E},\module{F})_\Total
            &\coloneqq
            \Hom(\module{E}_\Total, \module{F}_\Total)
            \\
            \CoisoHom(\module{E},\module{F})_\Wobs
            &\coloneqq
            \Hom(\module{E},\module{F})
            \\
            \CoisoHom(\module{E},\module{F})_\Null
            &\coloneqq
            \left\{
                (\Phi_\Total, \Phi_\Wobs)
                \in
                \Hom(\module{E},\module{F})
                \mid
                \Phi_\Wobs(\module{E}_\Wobs) \subseteq \module{F}_\Null
            \right\},
	\end{split}
    \end{equation}
    where
    $\iota \colon \Hom(\module{E},\module{F}) \to
    \Hom(\module{E}_\Total, \module{F}_\Total)$ is the projection onto
    the first component.
\end{remarklist}
Using the image every morphism of coisotropic modules can be
factorized into a regular epimorphism and a monomorphism.
\begin{remarklist}[resume,label=\alph*)]
\item The coproduct of $\module{E}$ and $\module{F}$ is given by
    \begin{equation}
	\module{E} \oplus \module{F}
	=
        \left(
            \module{E}_\Total \oplus \module{F}_\Total,
            \;
            \module{E}_\Wobs \oplus \module{F}_\Wobs,
            \;
            \module{E}_\Null \oplus \module{F}_\Null
        \right)
    \end{equation}
    with $\iota_\oplus = \iota_\module{E} + \iota_\module{F}$. It is
    also called the \emph{direct sum} of $\module{E}$ and
    $\module{F}$.
\end{remarklist}
It should be clear that also infinite direct sums can be defined this
way.

The definition of coisotropic modules allows us to reinterpret several
(geometric) reduction procedures in a completely algebraic fashion, as
stated in the following straightforward proposition.
\begin{proposition}[Reduction]
    \label{prop:Reduction}%
    Mapping a coisotropic module $\module{E}$ to the quotient
    $\module{E}_\red = \module{E}_\Wobs / \module{E}_\Null$ and
    morphisms of coisotropic modules to the induced morphisms yields a
    functor
    \begin{equation}
        \red\colon
        \CoisoModTriple \to \Modules,
    \end{equation}
    to the category $\Modules$ of $\field{k}$-modules.
\end{proposition}

The coisotropic modules that arise in geometric examples will be
equipped with an additional multiplicative structure.  This is
captured by the following definition.
\begin{definition}[Coisotropic algebra]
    \label{def:CoisotropicAlgebras}%
    \begin{definitionlist}
    \item A \emph{coisotropic algebra} is a triple
        $\algebra{A} = (\algebra{A}_\Total, \algebra{A}_\Wobs,
        \algebra{A}_\Null)$ consisting of unital associative
        $\field{k}$-algebras $\algebra{A}_\Total$, $\algebra{A}_\Wobs$
        and a two-sided ideal
        $\algebra{A}_\Null \subseteq \algebra{A}_\Wobs$ together with
        a unital algebra homomorphism
        $\iota \colon \algebra{A}_\Wobs \to \algebra{A}_\Total$.
    \item A morphism of coisotropic algebras $\algebra{A}$ and
        $\algebra{B}$ is given by a pair of unital algebra
        homomorphisms
        $\Phi_\Total \colon \algebra{A}_\Total \to \algebra{B}_\Total$
        and
        $\Phi_\Wobs \colon \algebra{A}_\Wobs \to \algebra{B}_\Wobs$
        such that
        $\iota_\algebra{B} \circ \Phi_\Wobs = \Phi_\Total \circ
        \iota_\algebra{A}$ and
        $\Phi_\Wobs(\algebra{A}_\Null) \subseteq \algebra{B}_\Null$.
    \item The category of coisotropic algebras is denoted by
        $\CoisoAlgTriple$.
    \end{definitionlist}
\end{definition}
Note that $\CoisoAlgTriple$ forms indeed a category. A coisotropic
algebra $\algebra{A}$ is called \emph{commutative} if
$\algebra{A}_\Total$ and $\algebra{A}_\Wobs$ are commutative algebras.
The reduction functor of \autoref{prop:Reduction} clearly induces a
functor $\red\colon \CoisoAlgTriple \to \Algebras$ with the reduced
algebra given by
$\algebra{A}_\red = \algebra{A}_\Wobs / \algebra{A}_\Null$.
\begin{example}
    \label{ex:CoisotropicAlgebras}%
    From differential geometry we obtain the following two basic
    examples:
    \begin{examplelist}
    \item \label{ex:CoisotropicAlgebras_1} Let $\iota \colon C \to M$
        be a submanifold of a manifold $M$ and let $D \subseteq TC$ be
        an integrable distribution on $C$.  We denote the functions on
        $C$ constant along the leaves of $D$ by $\Cinfty_D(C)$.  Then
        \begin{equation}
            \Cinfty(M,C,D)
            \coloneqq
            (\Cinfty(M), \Cinfty_D(M), \algebra{J}_C),
        \end{equation}
        with
        \begin{equation}
            \Cinfty_D(M)
            \coloneqq
            \lbrace f \in \Cinfty(M) \mid \iota^* f \in \Cinfty_D(C) \rbrace
        \end{equation}
        and the vanishing ideal
        \begin{equation}
            \algebra{J}_C
            \coloneqq
            \lbrace f \in \Cinfty(M) \mid \iota^* f = 0 \rbrace
        \end{equation}
        is a commutative coisotropic algebra.  As soon as the leaf
        space $C/D$ carries a canonical manifold structure we have
        $\Cinfty(M,C,D)_\red = \Cinfty(C/D)$.  Note the slight abuse of
        notation: $\Cinfty_D(C)$ denotes the algebra of functions on
        $C$ which are constant along the distribution on $C$, whereas
        $\Cinfty_D(M)$ denotes the algebra of functions on $M$ which
        are constant along the distribution only on the submanifold
        $C$.
    \item \label{ex:CoisotropicAlgebras_2} Let $(M,\pi)$ be a Poisson
        manifold together with a coisotropic submanifold
        $\iota\colon C \to M$.  Then
        $\algebra{A} = (\Cinfty(M), \Cinfty_D(M), \algebra{J}_C)$ is a
        commutative coisotropic algebra and
        $\algebra{A}_\red \simeq \Cinfty_D(M) / \algebra{J}_C$ is even
        a Poisson algebra.
    \end{examplelist}
\end{example}

Having these examples in mind we want to understand vector bundles
that are compatible with the submanifold and the distribution on it.
On the algebraic side this will be captured by the following notion of
module over a coisotropic algebra.
\begin{definition}[Right module over coisotropic algebra]
    \label{def:CoisoModules}%
    Let $\algebra{A} \in \CoisoAlgTriple$ be a coisotropic algebra.
    \begin{definitionlist}
    \item \label{item:CoisoTripleBimodule} A triple
        $\module{E} = (\module{E}_\Total, \module{E}_\Wobs,
        \module{E}_\Null)$ consisting of a right
        $\algebra{A}_\Total$-module $\module{E}_\Total$ and right
        $\algebra{A}_\Wobs$-modules $\module{E}_\Wobs$ and
        $\module{E}_\Null$ together with a module morphism
        $\iota_\module{E} \colon \module{E}_\Wobs \longrightarrow
        \module{E}_\Total$ along the morphism
        $\iota_\algebra{A} \colon \algebra{A}_\Wobs \to
        \algebra{A}_\Total$ is called a \emph{coisotropic right
          $\algebra{A}$-module} if
        $\module{E}_\Null \subseteq \module{E}_\Wobs$ is a sub-module
        such that
        \begin{equation}
            \label{eq:ModuleTripleCondition}
            \module{E}_\Wobs \cdot \algebra{A}_\Null
            \subseteq
            \module{E}_\Null.
        \end{equation}
    \item \label{item:CoisoTripleBimoduleMorphism} A morphism
        $\Phi\colon \module{E} \longrightarrow \module{F}$ between
        coisotropic right $\algebra{A}$-modules is a pair
        $(\Phi_\Total, \Phi_\Wobs)$ consisting of an
        $\algebra{A}_\Total$-module morphism
        $\Phi_\Total \colon \module{E}_\Total \longrightarrow
        \module{F}_\Total$ and an $\algebra{A}_\Wobs$-module morphism
        $\Phi\colon \module{E}_\Wobs \longrightarrow \module{F}_\Wobs$
        such that
        $\Phi_\Total \circ \iota_\module{E} = \iota_{\module{F}} \circ
        \Phi_\Wobs$ and
        $\Phi_\Wobs(\module{E}_\Null) \subseteq \module{F}_\Null$.
    \item The category of coisotropic right $\algebra{A}$-modules is
        denoted by $\cCoisoModTriple{\algebra{A}}$.
    \end{definitionlist}
\end{definition}
There is an obvious notion of left modules and bimodules over
coisotropic algebras, see \cite{dippell.esposito.waldmann:2019a}, but
in the following we will only need right modules.  Therefore we will
shorten our notation and just say module instead of right module from
now on.  If we consider the coisotropic algebra
$\coisoField{k} = (\field{k}, \field{k}, 0)$, then modules over the
coisotropic algebra $\coisoField{k}$ agree with coisotropic
$\field{k}$-modules as introduced in
\autoref{definition:CoisoTriplesModules}, i.e.
$\cCoisoModTriple{\coisoField{k}} = \CoisoModTriple$.
\begin{example}
    Let $C \subseteq M$ be a submanifold and let $D \subseteq TC$ be
    an integrable distribution on $C$.  Let, moreover, $E \to M$ be a
    vector bundle over $M$ and $\nabla$ a covariant derivative on $E$.
    Then setting $\module{E}_\Total = \Secinfty(E)$,
    \begin{equation}
        \label{eq:ENormal}
        \module{E}_\Wobs
        =
        \left\{
            s \in \Secinfty(E)
            \; \big| \;
            \nabla_X s\at{C} = 0
            \text{ for all }
            X \in \Secinfty(TM)
            \textrm{ with }
            X\at{C} \in D
        \right\}
    \end{equation}
    and $\module{E}_\Null = \{ s \in \Secinfty(E) \mid s\at{C} = 0 \}$
    defines a coisotropic $\algebra{A}$-module $\module{E}$ for
    $\algebra{A} = (\Cinfty(M), \Cinfty_D(M), \algebra{J}_C)$ as in
    \autoref{ex:CoisotropicAlgebras}, \ref{ex:CoisotropicAlgebras_1}.
    Note that the construction of $\module{E}_\Wobs$ strongly depends
    on the choice of the covariant derivative.
\end{example}

Clearly, the quotient $\module{E}_\Wobs / \module{E}_\Null$ is an
$\algebra{A}_\red$-module for any $\algebra{A}$-module $\module{E}$.
Thus, we can easily extend the reduction of \autoref{prop:Reduction}
to the case of modules over coisotropic algebras by constructing the
functor
$\red \colon \cCoisoModTriple{\algebra{A}} \to
\cModules{\algebra{A}_\red}$ using
$\module{E}_\red = \module{E}_\Wobs / \module{E}_\Null$.

%
% Coisotropic Sets
%

\subsection{Coisotropic Index Sets}
\label{sec:CoisotropicIndexSets}

Forgetting all algebraic structure of a given coisotropic module or
coisotropic algebra yields a triple of sets
$(M_\Total,M_\Wobs,M_\Null)$ such that $M_\Null \subseteq M_\Wobs$ and
a map $\iota_M \colon M_\Wobs \to M_\Total$.
\begin{definition}[Coisotropic index set]
    \begin{definitionlist}
    \item A triple $(M_\Total, M_\Wobs, M_\Null)$ of sets with
        $M_\Null \subseteq M_\Wobs$ together with a map
        $\iota_M \colon M_\Wobs \longrightarrow M_\Total$ is called a
        \emph{coisotropic index set}.
    \item A morphism of coisotropic index sets
        $M = (M_\Total, M_\Wobs ,M_\Null)$ and
        $N =(N_\Total, N_\Wobs, N_\Null)$ is a pair
        $(f_\Total, f_\Wobs)$ with
        $f_\Total \colon M_\Total \longrightarrow N_\Total$ and
        $f_\Wobs \colon M_\Wobs \longrightarrow N_\Wobs$ such that
        $f_\Wobs(M_\Null) \subseteq N_\Null$ and
        $f_\Total \circ \iota_M = \iota_N \circ f_\Wobs$.
    \item The category of coisotropic index sets is denoted by
        $\SetTriple$.
    \end{definitionlist}
\end{definition}
\begin{remark}\label{remark:Coisotropic_Sets}
    Instead of keeping the underlying set of the $\Null$-component of
    a given coisotropic module we could also use the equivalence
    relation on the $\Wobs$-component induced by the
    $\Null$-component.  This leads to the notion of coisotropic sets
    as introduced in \cite{dippell.esposito.waldmann:2020a}.  These
    two concepts do not agree, but for our purpose coisotropic index
    sets will be more useful.
\end{remark}

The category $\SetTriple$ inherits a lot of structure from the
category $\Sets$ of sets.  In particular, $\SetTriple$ has all finite
limits and colimits.  But it does not resemble $\Sets$ completely, as
the next result shows.
\begin{proposition}[Mono- and epimorphisms in $\SetTriple$]
    Let $f \colon M \to N$ be a morphism of coisotropic index sets.
    \begin{propositionlist}
    \item The morphism $f$ is a monomorphism iff $f_\Total$
        and $f_\Wobs$ are injective.
    \item The morphism $f$ is an epimorphism iff $f_\Total$
        and $f_\Wobs$ are surjective.
    \item The morphism $f$ is a regular monomorphism iff
        $f_\Total$ and $f_\Wobs$ are injective and
        $f_\Wobs^{-1}(N_\Null) = M_\Null$.
    \item The morphism $f$ is a regular epimorphism iff
        $f_\Total$ and $f_\Wobs$ are surjective and
        $f_\Wobs(M_\Null) = N_\Null$.
    \end{propositionlist}
\end{proposition}
\begin{remark}
    The relationship between the categories $\Sets$ and $\SetTriple$
    can be made more precise: In contrast to $\Sets$ the category
    $\SetTriple$ is not a topos, but only a quasi-topos, since it does
    only admit a weak subobject classifier.  See
    e.g. \cite{wyler:1991a} and \cite{johnstone:2002a} for more on
    quasi-topoi.
\end{remark}

Recall, that in $\Sets$ the axiom of choice can be rephrased by saying
that every epimorphism splits, i.e. for every surjective map
$f \colon M \to N$ there exists a map $g \colon N \to M$ such that
$f \circ g = \id_M$.  The next example shows that the equivalent
statement need not be true in $\SetTriple$.  In particular, not even
every regular epimorphism splits.
\begin{example}
    Let $M = (\{0,1\}, \{0,1\}, \emptyset)$ with $\iota_M$ the
    identity, and $N = (\{0\},\{0,1\},\emptyset)$ with $\iota_N = 0$.
    Then $\Phi = (0, \id_{\{0,1\}})\colon M \to N$ is a regular
    epimorphism.  Suppose there exists a split $\Psi \colon N \to M$,
    then $\Psi_\Wobs = \id_{\{0,1\}}$.  But $\Psi_\Total$ would need
    to fulfill both
    $\Psi_\Total(0) = \Psi_\Total(\iota_N(0)) = \iota_M(\Psi_\Wobs(0))
    = 0$ and
    $\Psi_\Total(0) = \Psi_\Total(\iota_N(1)) = \iota_M(\Psi_\Wobs(1))
    = 1$. Therefore $(\Psi_\Total,\Psi_\Wobs)$ can not be chosen to be
    a morphism of triples of sets.
\end{example}

The next result characterizes completely the coisotropic index sets
for which any regular epimorphism into them splits.
\begin{proposition}
    \label{prop:ProjectiveIndexSets}%
    Let $P \in \SetTriple$ be a coisotropic index set.  Then the
    following statements are equivalent:
    \begin{propositionlist}
    \item\label{prop:ProjectiveIndexSets_1} Every regular epimorphism
        $M \to P$ splits.
    \item\label{prop:ProjectiveIndexSets_2} For every regular
        epimorphism $\Phi \colon M \to N$ and every morphism
        $\Psi \colon P \to N$ there exists a morphism
        $\chi \colon P \to M$ such that $\Phi \circ \chi = \Psi$.
    \item\label{prop:ProjectiveIndexSets_3} The map
        $\iota_P \colon P_\Wobs \to P_\Total$ is injective.
    \end{propositionlist}
\end{proposition}
\begin{proof}
    For \ref{prop:ProjectiveIndexSets_1} $\implies$
    \ref{prop:ProjectiveIndexSets_2}, suppose every regular
    epimorphism into $P$ splits.  Let $\Phi \colon M \to N$ and
    $\Psi \colon P \to N$ be given, with $\Phi$ a regular epimorphism.
    We can construct the pullback $P \deco{}{\Psi}{\times}{}{\Phi} M$
    of coisotropic index sets given by
    \begin{align*}
        (P \deco{}{\Psi}{\times}{}{\Phi} M)_\Total
        &=
        \left\{
            (x,y) \in P_\Total \times M_\Total
            \mid
            \Psi_\Total(x) = \Phi_\Total(y)
        \right\}
        \\
        (P \deco{}{\Psi}{\times}{}{\Phi} M)_\Wobs
        &=
        \left\{
            (x,y) \in P_\Wobs \times M_\Wobs
            \mid
            \Psi_\Wobs(x) = \Phi_\Wobs(y)
        \right\}
        \\
        (P \deco{}{\Psi}{\times}{}{\Phi} M)_\Null
        &=
        \left\{
            (x,y) \in P_\Null \times M_\Null
            \mid
            \Psi_\Wobs(x) = \Phi_\Wobs(y)
        \right\}
    \end{align*}
    with the morphism $\iota_P \times \iota_M$.  Then it is easy to
    see that $\pr_P \colon P \deco{}{\Psi}{\times}{}{\Phi} M \to P$ is
    a regular epimorphism.  Thus by assumption there exists a split
    $\tilde{\chi} \colon P \to P \deco{}{\Psi}{\times}{}{\Phi} M$, and
    thus $\chi = \mathord{\pr_M} \circ \tilde{\chi}$ fulfills
    $\Phi \circ \chi = \Psi$ as wanted. Next, we consider
    \ref{prop:ProjectiveIndexSets_2} $\implies$
    \ref{prop:ProjectiveIndexSets_3}: Construct a new coisotropic
    index set $M$ by defining
    $M_\Total = P_\Wobs \sqcup (P_\Total \setminus \image(\iota_P))$,
    $M_\Wobs = P_\Wobs$, $M_\Null = P_\Null$ and $\iota_M$ by the
    obvious inclusion.  Then $\iota_M$ is injective and
    $\Phi \colon M \to P$ defined by $\Phi_\Wobs = \id_{P_\Wobs}$ and
    \begin{equation*}
	\Phi_\Total(p)
        =
	\begin{cases}
            \iota_P(p) & \text{ if } p \in P_\Wobs \\
            p & \text{ if } p \in P_\Total\setminus \image(\iota_P)
	\end{cases}
    \end{equation*}
    is a regular epimorphism.  Hence there exists a morphism
    $\chi \colon P \to M$ such that $\Phi \circ \chi = \id_P$.  But
    then $\chi_\Wobs$ is injective and because of
    $\chi_\Total \circ \iota_P = \iota_M \circ \chi_\Wobs$ this
    implies the injectivity of $\iota_P$.  Finally, we show
    \ref{prop:ProjectiveIndexSets_3} $\implies$
    \ref{prop:ProjectiveIndexSets_1}: Let
    $\Phi \colon M \longrightarrow P$ be a regular epimorphism in
    $\SetTriple$, i.e.  $\Phi_\Total$ and $\Phi_\Wobs$ are surjective
    and $\Phi_\Wobs(M_\Null) = P_\Null$.  Thus
    $\Phi_\Null \colon M_\Null \longrightarrow P_\Null$ is a
    surjective map of sets.  Hence there exists a map
    $\Psi_\Null \colon P_\Null \longrightarrow M_\Null$ such that
    $\Phi_\Null \circ \Psi_\Null = \id_{P_\Null}$.  This map can now
    clearly be extended to $P_\Wobs$, giving
    $\Psi_\Wobs \colon P_\Wobs \longrightarrow M_\Wobs$ with
    $\Phi_\Wobs \circ \Psi_\Wobs = \id_{P_\Wobs}$ preserving the
    $\Null$-component.
    \begin{equation*}
	\begin{tikzcd}
            M_\Total
            \arrow[twoheadrightarrow]{r}{\Phi_\Total}
            & P_\Total
            \arrow[bend left = 20]{d}{\tau}
            \\
            M_\Wobs
            \arrow{u}{\iota_M}
            \arrow[twoheadrightarrow]{r}{\Phi_\Wobs}
            & P_\Wobs
            \arrow{u}{\iota_P}
            \arrow[bend left = 20]{l}{\Psi_\Wobs}
	\end{tikzcd}
    \end{equation*}
    Now if we restrict the codomain of $\iota_P$ to $\image(\iota_P)$
    then $\iota_P \colon P_\Wobs \longrightarrow \image(\iota_P)$ is
    surjective, hence there exists a split
    $\tau \colon \image(\iota_P) \longrightarrow P_\Wobs$.  This
    allows us to define
    $\tilde\Psi_\Total = \iota_M \circ \Psi_\Wobs \circ \tau$.  This
    is a split of $\Phi_\Total\at{\image \iota_M}$, since
    \begin{equation*}
	\Phi_\Total \circ \tilde\Psi_\Total
	=
        \Phi_\Total \circ \iota_M \circ \Psi_\Wobs \circ \tau
	=
        \iota_P \circ \Phi_\Wobs \circ \Psi_\Wobs \circ \tau
	=
        \id_{\image(\iota_P)}.
    \end{equation*}
    Then extending $\tilde\Psi_\Total$ to the whole of $P_\Total$
    yields a section $\Psi_\Total \colon P_\Total \to M_\Total$ of
    $\Phi_\Total$.  To show that this is now a morphism of coisotropic
    index sets we need the injectivity of $\iota_P$.  We have
    \begin{equation*}
	\Psi_\Total \circ \iota_P
	=
        \iota_M \circ \Psi_\Wobs \circ \tau \circ \iota_P
	=
        \iota_M \circ \Psi_\Wobs
    \end{equation*}
    since $\tau \circ \iota_P = \id_{P_\Wobs}$.
\end{proof}

We will denote the category of coisotropic index sets $P$ with
injective $\iota_P$ by $\SetTripleInj$.
\autoref{prop:ProjectiveIndexSets} shows that $\SetTripleInj$ is
exactly the subcategory of regular objects in $\SetTriple$.

%
% Free Coisotropic Modules
%

\subsection{Free Coisotropic Modules}
\label{sec:FreeCoisotropicModules}

After having established a reasonable notion of sets underlying
coisotropic modules we now want to construct free coisotropic modules
with basis given by a coisotropic index set.
\begin{lemma}
    Let $M \in \SetTriple$ and $\algebra{A} \in \CoisoAlgTriple$ be
    given.  Then $\algebra{A}^{(M)}$ defined by
    \begin{align}
        (\algebra{A}^{(M)})_\Total
        &\coloneqq
        \algebra{A}_\Total^{(M_\Total)}
        \\
        (\algebra{A}^{(M)})_\Wobs
        &\coloneqq
        \algebra{A}_\Wobs^{(M_\Wobs)}
        \\
        (\algebra{A}^{(M)})_\Null
        &\coloneqq
        \algebra{A}_\Wobs^{(M_\Null)} + \algebra{A}_\Null^{(M_\Wobs)}
    \end{align}
    with
    $\iota_{\algebra{A}^{(M)}} \colon \algebra{A}_\Wobs^{(M_\Wobs)}
    \to \algebra{A}_\Total^{(M_\Total)}$ given by
    \begin{equation}
        \iota_{\algebra{A}^{(M)}}
        \Biggl(\sum_{m \in M_\Wobs} b_m^{\Wobs} a_m\Biggr)
        \coloneqq
        \sum_{m \in M_\Total} b_m^{\Total}
        \Biggl(
        \sum_{n \in \iota_M^{-1}(m)} \iota_\algebra{A}(a_n)
        \Biggr)
    \end{equation}
    is a coisotropic $\algebra{A}$-module.  Here $b_m^\Total$ and
    $b_m^\Wobs$ denote the standard bases of
    $\algebra{A}_\Total^{(M_\Total)}$ and
    $\algebra{A}_\Wobs^{(M_\Wobs)}$, respectively.
\end{lemma}

The following result shows that these coisotropic modules fulfill the
usual universal property for free modules.
\begin{proposition}
    Let $M \in \SetTriple$ and $\algebra{A} \in \CoisoAlgTriple$ be
    given.  The coisotropic $\algebra{A}$-module $\algebra{A}^{(M)}$
    together with the morphism
    $i \colon M \longrightarrow \algebra{A}^{(M)}$ given by the usual
    embedding of the standard basis fulfills the following universal
    property: For every $\module{E} \in \cCoisoModTriple{\algebra{A}}$
    and every morphism $\phi \colon M \longrightarrow \module{E}$ of
    coisotropic index sets there exists a unique morphism
    $\Phi \colon \algebra{A}^{(M)} \longrightarrow \module{E}$ of
    coisotropic $\algebra{A}$-modules such that the diagram
    \begin{equation}
        \begin{tikzcd}
            \algebra{A}^{(M)}
            \arrow[dashed]{r}{\Phi}
            & \module{E}
            \\
            M
            \arrow{u}{i}
            \arrow{ur}[swap]{\phi}
            & { }
        \end{tikzcd}
    \end{equation}
    commutes.
\end{proposition}
\begin{proof}
    Note that $(\algebra{A}^{(M)})_\Total$ and
    $(\algebra{A}^{(M)})_\Wobs$ are free.  Therefore, they fulfill
    the corresponding universal property separately.  Moreover, we
    clearly have
    \begin{equation*}
	\Phi_\Wobs(\algebra{A}_\Null^{(M_\Wobs)})
	=
        \Phi_\Wobs(\algebra{A}_\Wobs^{(M_\Wobs)} \cdot \algebra{A}_\Null)
	=
        \Phi_\Wobs(\algebra{A}_\Wobs^{(M_\Wobs)}) \cdot \algebra{A}_\Null
	\subseteq
        \module{E}_\Wobs \cdot \algebra{A}_\Null
	\subseteq
        \module{E}_\Null
    \end{equation*}
    and
    $\Phi_\Wobs(\algebra{A}_\Wobs^{(M_\Null)})
    = \Phi_\Wobs(i_\Wobs(M_\Null)) =
    \phi_\Wobs(M_\Null) \subseteq \module{E}_\Null$.
\end{proof}
\begin{definition}[$\SetTriple$-free coisotropic $\algebra{A}$-module]
    A coisotropic $\algebra{A}$-module of the form $\algebra{A}^{(M)}$
    for some coisotropic index set $M \in \SetTriple$ is called
    \emph{$\SetTriple$-free} with basis $M$.  It is called
    \emph{finitely generated} if $M$ consists of finite sets.
\end{definition}

Coisotropic modules of the form $\algebra{A}^{(M)}$ with $M$ a
coisotropic index set with injective
$\iota_M \colon M_\Wobs \to M_\Total$ will be called
$\SetTripleInj$-free.
\begin{remark}
    It is not hard to show that mapping coisotropic index sets to free
    coisotropic modules defines a functor from $\SetTriple$ to
    $\cCoisoModTriple{\algebra{A}}$ which is left adjoint to the
    forgetful functor from $\cCoisoModTriple{\algebra{A}}$ to
    $\SetTriple$, hence justifying the name free coisotropic module.
\end{remark}
\begin{example}
    \label{Example:Free_Modules}%
    \begin{examplelist}
    \item \label{item:AlgebraIsFreeModule} Every coisotropic algebra
        $\algebra{A}$ regarded as a module over itself is free with
        basis $M = (\{\pt\}, \{\pt \}, \emptyset)$.
    \item \label{item:DualModuleNotFree} Analogously to
        \eqref{Eq:Hom_Module} we can associate to every coisotropic
        right $\algebra{A}$-module $\module{E}$ a dual $\module{E}^*$
        with
        \begin{equation}
            \begin{split}
                \module{E}^*_\Total
                &=
                \Hom_{\algebra{A}_\Total}
                \bigl(\module{E}_\Total, \algebra{A}_\Total\bigr),
                \\
                \module{E}^*_\Wobs
                &=
                \Hom_{\algebra{A}}\bigl(\module{E}, \algebra{A}\bigr),
                \\
                \module{E}^*_\Null
                &=
                \bigl\lbrace
                (\alpha_\Total, \alpha_\Wobs)
                \in
                \Hom_{\algebra{A}}\bigl(\module{E}, \algebra{A}\bigr)
                \bigm\vert
                \alpha_\Wobs\bigl(\module{E}_\Wobs\bigr)
                \subseteq
                \algebra{A}_\Null
                \bigr\rbrace.
            \end{split}
        \end{equation}
        It has the structure of a coisotropic left
        $\algebra{A}$-module.  Let now $\algebra{A}^{(M)}$ be a
        finitely generated $\SetTripleInj$-free coisotropic module
        with $M_\Null \subseteq M_\Wobs \subseteq M_\Total$.  Then its
        dual is given by
        \begin{equation}
            \begin{split}
                (\algebra{A}^{(M)})^*_\Total
                &\simeq
                \algebra{A}_\Total^{(M_\Total)},
                \\
                (\algebra{A}^{(M)})^*_\Wobs
                &\simeq
                \algebra{A}_\Null^{(M_\Null)}
                \oplus
                \algebra{A}_\Wobs^{(M_\Wobs \setminus M_\Null)}
                \oplus
                \algebra{A}_\Total^{(M_\Total \setminus M_\Wobs)},
                \\
                (\algebra{A}^{(M)})^*_\Null
                &\simeq
                \algebra{A}_\Null^{(M_\Wobs)}
                \oplus
                \algebra{A}_\Total^{(M_\Total \setminus M_\Wobs)}.
            \end{split}
        \end{equation}
        As the $\Wobs$-component might fail to be a free
        $\algebra{A}_\Wobs$-module, this is in general not a free
        coisotropic $\algebra{A}$-module.  Note, however, that its
        reduction is given by the free module
        $(\algebra{A}_\red)^{(M_\Wobs \setminus M_\Null)}$ and that
        there exists a direct sum decomposition
        \begin{equation}
            (\algebra{A}^{(M)})^*
            =
            \module{F} \oplus \module{G}
        \end{equation}
        into coisotropic left $\algebra{A}$-modules $\module{F}$ and
        $\module{G}$ where $\module{F}$ is free and
        $\module{G}_\red = \lbrace 0 \rbrace$.
    \end{examplelist}
\end{example}

%
% Projective Coisotropic Modules
%

\subsection{Projective Coisotropic Modules}
\label{sec:PorjectiveCoisotropicModules}

In this section we want to give characterizations of regular
projective coisotropic modules that more closely resemble the usual
characterizations of projective modules via projections, dual bases,
or as direct summands of free modules. For this let us start with a
categorical definition of projective objects using a lifting property
based on regular epimorphisms instead of general epimorphisms:
\begin{definition}[Regular projective module]
    A coisotropic $\algebra{A}$-module
    $\module{P} \in \cCoisoModTriple{\algebra{A}}$ is called
    \emph{regular projective} if for every morphism
    $\Psi \colon \module{P} \longrightarrow \module{F}$ and every
    regular epimorphism
    $\Phi \colon \module{E} \longrightarrow \module{F}$ there exists a
    morphism $\chi \colon \module{P} \longrightarrow \module{E}$ such
    that $\Phi \circ \chi = \Psi$.  Diagrammatically:
    \begin{equation}
	\begin{tikzcd}
            { }
            & \module{E}
            \arrow[twoheadrightarrow]{d}{\Phi} \\
            \module{P}
            \arrow{r}{\Psi}
            \arrow[dashed]{ur}{\chi}
            & \module{F}
	\end{tikzcd}
    \end{equation}
\end{definition}
We start with an important class of regular projective coisotropic
modules: the $\SetTripleInj$-free coisotropic modules are regular
projective.
\begin{lemma}
    Every $\SetTripleInj$-free coisotropic module in
    $\cCoisoModTriple{\algebra{A}}$ is regular projective.
\end{lemma}
\begin{proof}
    Let $\algebra{A}^{(M)}$ be a free coisotropic module with
    $M \in \SetTripleInj$.  Suppose the following morphisms are given:
    \begin{equation*}
	\begin{tikzcd}
            { }
            & \module{E}
            \arrow[twoheadrightarrow]{d}{\Phi} \\
            \algebra{A}^{(M)}
            \arrow{r}{\Psi}
            & \module{F}
	\end{tikzcd}
    \end{equation*}
    with $\Phi$ a regular epimorphism.  Since $\Phi$ and $\Psi$ induce
    morphisms $\phi \colon \module{E} \to \module{F}$ and
    $\psi \colon M \to \module{F}$ of coisotropic index sets we know
    by \autoref{prop:ProjectiveIndexSets} that there exists
    $\xi \colon M \to \module{E}$ such that $\phi \circ \xi = \psi$.
    Then by the freeness of $\algebra{A}^{(M)}$ there exists
    $\Xi \colon \algebra{A}^{(M)} \to \module{E}$ such that
    $\Phi \circ \Xi$ restricted to $M$ is just $\psi$.  Hence
    $\Phi \circ \Xi = \Psi$.
\end{proof}

The category $\cCoisoModTriple{\algebra{A}}$ has enough
$\SetTripleInj$-projectives in the following sense:
\begin{proposition}
    \label{prop:ModulesAreFreeImages}%
    For every coisotropic module
    $\module{E} \in \cCoisoModTriple{\algebra{A}}$ there exists
    $M \in \SetTripleInj$ and a regular epimorphism
    $\Phi \colon \algebra{A}^{(M)} \longrightarrow \module{E}$.
\end{proposition}
\begin{proof}
    As constructed in the proof of \autoref{prop:ProjectiveIndexSets}
    there exists $M \in \SetTripleInj$ and a regular epimorphism
    $\phi \colon M \to \module{E}$ of coisotropic index sets.  Then by
    the universal property of $\algebra{A}^{(M)}$ there exists
    $\Phi \colon \algebra{A}^{(M)} \longrightarrow \module{E}$ such
    that $\Phi \circ i = \phi$.  Then $\Phi$ is a regular epimorphism
    since so is $\phi$.
\end{proof}

Another important notion in the characterization of regular projective
coisotropic modules is that of a split exact sequence.  A sequence of
morphisms
\begin{equation}
    \begin{tikzcd}
        0
        \arrow{r}{}
        & A
        \arrow{r}{f}
        & B
        \arrow{r}{g}
        & C
        \arrow{r}{}
        & 0
    \end{tikzcd}
\end{equation}
is a called \emph{short exact} if $f$ is a monomorphism,
$\image(f) = \ker(g)$, and $g$ is a regular epimorphism.  It is called
\emph{split exact} if in addition there exists $h \colon C \to B$ such
that $g \circ h = \id_C$.
\begin{proposition}
    \label{prop:RegProjGiveSplit}%
    Let $\module{P} \in \cCoisoModTriple{\algebra{A}}$ be a regular
    projective coisotropic module.  Then every short exact sequence of
    the form
    \begin{equation}
	\begin{tikzcd}
            0
            \arrow{r}{}
            &\module{E}
            \arrow[hookrightarrow]{r}{\Phi}
            &\module{F}
            \arrow[twoheadrightarrow]{r}{\Psi}
            &\module{P}
            \arrow{r}{}
            &0
	\end{tikzcd}
    \end{equation}
    is a split exact sequence.
\end{proposition}
\begin{proof}
    Since $\module{P}$ is regular projective and $\Psi$ is a regular
    epimorphism the sequence splits by the universal property of $\module{P}$.
\end{proof}

For us split exact sequences are important because of the so-called
splitting lemma.  The splitting lemma is a basic result in homological
algebra for modules over rings. Despite
$\cCoisoModTriple{\algebra{A}}$ not being an abelian category, the
splitting lemma nevertheless holds for coisotropic modules.
\begin{proposition}[Splitting lemma in $\cCoisoModTriple{\algebra{A}}$]
    A short exact sequence
    \begin{equation}
	\label{prop:SplittingLemma_shortexseq}
	\begin{tikzcd}
            0
            \arrow{r}{}
            &\module{E}
            \arrow[hookrightarrow]{r}{\Phi}
            &\module{F}
            \arrow[twoheadrightarrow]{r}{\Psi}
            &\module{G}
            \arrow{r}{}
            &0
	\end{tikzcd}
    \end{equation}
    in $\cCoisoModTriple{\algebra{A}}$ splits if and only if it is
    isomorphic as a sequence to
    \begin{equation}
	\begin{tikzcd}
            0
            \arrow{r}{}
            &\module{E}
            \arrow[hookrightarrow]{r}{i_\module{E}}
            &\module{E} \oplus \module{G}
            \arrow[twoheadrightarrow]{r}{p_\module{G}}
            &\module{G}
            \arrow{r}{}
            &0
	\end{tikzcd}
    \end{equation}
    with the canonical inclusion $i_\module{E}$ and projection
    $p_\module{G}$.
\end{proposition}
\begin{proof}
    Suppose there exists
    $\xi \colon \module{F} \longrightarrow \module{E}$ such that
    $\xi \circ \Phi = \id_\module{E}$.  Then we know that
    $\module{F}_\Total \simeq \module{E}_\Total \oplus
    \module{G}_\Total$ and
    $\module{F}_\Wobs \simeq \module{E}_\Wobs \oplus \module{G}_\Wobs$
    by the splitting lemma in the respective categories of modules.
    We denote these isomorphisms by $\theta_\Total$ and
    $\theta_\Wobs$, respectively.  To show that these form a morphism
    of coisotropic triples consider that
    $\theta = i_\module{E} \circ \xi + i_\module{G} \circ \Psi$ is a
    composition of morphisms of coisotropic triples.  Thus $\theta$
    itself is a morphism of coisotropic triples.  Moreover,
    $\theta_\Wobs(\module{F}_\Null) = \xi_\Wobs(\module{F}_\Null) +
    \Psi_\Wobs(\module{F}_\Null) = \module{E}_\Null
    + \module{G}_\Null$
    holds since $\xi$ and $\Psi$ are regular epimorphisms.  Hence,
    $\theta$ is an isomorphism of coisotropic modules.  Conversely,
    suppose
    $\theta \colon \module{F} \longrightarrow \module{E} \oplus
    \module{G}$ is an isomorphism such that
    $\theta \circ \Phi = i_\module{E}$ and
    $p_\module{G} \circ \theta = \Psi$.  Then
    $\theta^{-1} \circ i_\module{G}$ is clearly a splitting for
    \eqref{prop:SplittingLemma_shortexseq}.
\end{proof}

The following result shows that regular projective coisotropic modules
can be described as direct summands of $\SetTripleInj$-free modules.
The proof is completely analogous to the usual case, see e.g.
\cite[Prop. 3.10]{jacobson:1989a}.
\begin{theorem}[Regular projective modules]
    \label{thm:RegularProjective}%
    Let $\module{P} \in \cCoisoModTriple{\algebra{A}}$ be given.  The
    following statements are equivalent:
    \begin{theoremlist}
    \item \label{thm:RegularProjective_1} The module $\module{P}$ is
        regular projective.
    \item \label{thm:RegularProjective_2} Every short exact sequence
        $0 \rightarrow \module{E} \rightarrow \module{F} \rightarrow
        \module{P} \rightarrow 0$ splits.
    \item \label{thm:RegularProjective_3} The module $\module{P}$ is a
        direct summand of a $\SetTripleInj$-free module, i.e.  there
        exists $M \in \SetTripleInj$ and
        $\module{E} \in \cCoisoModTriple{\algebra{A}}$ such that
        $\algebra{A}^{(M)} \simeq \module{P} \oplus \module{E}$.
    \item \label{thm:RegularProjective_4} There exists
        $M \in \SetTripleInj$ and
        $e = (e_\Total, e_\Wobs) \in
        \End_\algebra{A}(\algebra{A}^{(M)})$ such that $e^2 = e$ and
        $\module{P} \simeq e\algebra{A}^{(M)} = \image(e)$.
    \end{theoremlist}
\end{theorem}
\begin{proof}
    We first show the equivalence of \ref{thm:RegularProjective_1},
    \ref{thm:RegularProjective_2} and \ref{thm:RegularProjective_3}.
    The implication \ref{thm:RegularProjective_1} $\implies$
    \ref{thm:RegularProjective_2} is given by
    \autoref{prop:RegProjGiveSplit}.  Assume
    \ref{thm:RegularProjective_2}.  By
    \autoref{prop:ModulesAreFreeImages} there exists a short exact
    sequence
    $0 \rightarrow \module{E} \rightarrow \algebra{A}^{(M)} \rightarrow
    \module{P} \rightarrow 0$ with $M \in \SetTripleInj$.  This
    sequence splits by assumption, and therefore by the splitting
    lemma we have
    $\algebra{A}^{(M)} \simeq \module{E} \oplus \module{P}$.  Now
    assuming \ref{thm:RegularProjective_3} we have a split exact
    sequence
    $0 \rightarrow \module{E} \rightarrow \algebra{A}^{(M)}
    \rightarrow \module{P} \rightarrow 0$ with $M \in \SetTripleInj$.
    Let $\Psi \colon \module{P} \longrightarrow \module{F}$ and
    $\Phi \colon \module{G} \longrightarrow \module{F}$ be given with
    $\Phi$ regular epimorphism.  We get the following diagram:
    \begin{equation*}
	\begin{tikzcd}
            0
            \arrow{r}{}
            &  \module{E}
            \arrow{r}{\iota}
            & \algebra{A}^{(M)}
            \arrow{r}{\pi}
            & \module{P}
            \arrow{r}{}
            \arrow{d}{\Psi}
            \arrow[bend left = 10, shift left = 2pt]{l}{\sigma}
            & 0 \\
            { }
            & { }
            & \module{G}
            \arrow[twoheadrightarrow]{r}{\Phi}
            & \module{F}
            & { }
	\end{tikzcd}
    \end{equation*}
    Since $\algebra{A}^{(M)}$ is regular projective there exists a
    morphism
    $\eta \colon \algebra{A}^{(M)} \longrightarrow \module{G}$ such
    that $\Phi \circ \eta = \Psi \circ \pi$.  Then
    $\eta \circ \sigma \colon \module{P} \longrightarrow \module{G}$
    yields the desired morphism making $\module{P}$ projective.
    Statements \ref{thm:RegularProjective_3} and
    \ref{thm:RegularProjective_4} are equivalent: If
    $\algebra{A}^{(M)} \simeq \module{P} \oplus \module{E}$, then
    choose for $e \in \End_\algebra{A}(\algebra{A}^{(M)})$ the
    projection on $\module{P}$.  If
    $\module{P} \simeq e\algebra{A}^{(M)}$, then
    $\module{E} \coloneqq \ker(e)$ gives the correct direct summand.
\end{proof}

The equivalence of \ref{thm:RegularProjective_1} and
\ref{thm:RegularProjective_2} can also be phrased as $\module{P}$ is
regular projective if and only if every regular epimorphism
$\module{F} \to \module{P}$ splits, compare
\autoref{prop:ProjectiveIndexSets}.  A coisotropic index set $M$ such
that $\algebra{A}^{(M)} \simeq \module{P} \oplus \module{E}$ is called
\emph{generating set} of the regular projective module $\module{P}$.
In addition to the above characterizations of regular projective
modules we can also use a coisotropic version of a dual basis.
\begin{proposition}[Dual basis]
    \label{prop:DualBasis}%
    Let $\module{P} \in \cCoisoModTriple{\algebra{A}}$.  Then
    $\module{P}$ is regular projective with generating set
    $M \in \SetTripleInj$ if and only if there exist families
    $(e_n)_{n \in M_\Total} \subseteq \module{P}_\Total$ and
    $(f_m)_{m \in M_\Wobs} \subseteq \module{P}_\Wobs$ of elements and
    families
    $(e^n)_{n \in M_\Total} \subseteq (\module{P}_\Total)^* =
    \Hom_{\algebra{A}_\Total}(\module{P}_\Total , \algebra{A}_\Total)$
    and
    $(f^m)_{m \in M_\Wobs} \subseteq (\module{P}_\Wobs)^*=
    \Hom_{\algebra{A}_\Wobs}(\module{P}_\Wobs , \algebra{A}_\Wobs)$ of
    functionals such that
    \begin{equation}
        \label{prop:DualBasis_Eq1}
        x_\Total = \sum_{n \in M_\Total} e_n e^n(x_\Total)
        \quad
        \textrm{and}
        \quad
        x_\Wobs = \sum_{m \in M_\Wobs} f_m f^m(x_\Wobs)
    \end{equation}
    for all
    $x_\Total \in \module{P}_\Total, x_\Wobs \in \module{P}_\Wobs$
    where for fixed $x_\Total,x_\Wobs$ only finitely many of the
    $e^n(x_\Total), f^m(x_\Wobs)$ differ from $0$.  Moreover, the
    following properties need to be satisfied:
    \begin{propositionlist}
    \item \label{prop:DualBasis_1} One has
        $e_{\iota_M(m)} = \iota_{\module{P}}(f_m)$ for
        $m \in M_\Wobs$.
    \item \label{prop:DualBasis_2} One has $f_m \in \module{P}_\Null$
        for $m \in M_\Null$.
    \item \label{prop:DualBasis_3} One has
        $e^{\iota_M(m)} \circ \iota_{\module{P}} = \iota_\algebra{A}
        \circ f^m$ for $m \in M_\Wobs$.
    \item \label{prop:DualBasis_4} One has
        $e^n \circ \iota_{\module{P}} = 0$ for
        $n \in M_\Total \setminus \iota_M(M_\Wobs)$.
    \item \label{prop:DualBasis_5} One has
        $f^m(x) \in \algebra{A}_\Null$ for $x \in \module{P}_\Null$,
        $m \in M_\Wobs \setminus M_\Null$.
    \end{propositionlist}
\end{proposition}
\begin{proof}
    Let $\module{P} \simeq e\algebra{A}^{(M)}$ be regular projective
    with idempotent $e \in \End_\algebra{A}(\algebra{A}^{(M)})$ and
    generating set $M \in \SetTripleInj$.  Denote by
    $b_n \in \algebra{A}_\Total^{(M_\Total)}$ and
    $c_m \in \algebra{A}_\Wobs^{(M_\Wobs)}$ the standard bases and by
    $b^n$ and $c^m$ the canonical coordinate functionals.  Defining
    $e_n = e_\Total(b_n)$ and $f_m = e_\Wobs(c_m)$ for
    $n \in M_\Total$ and $m \in M_\Wobs$ as well as
    $e^n = b^n\at{e\algebra{A}^{M}}$ and
    $f^m = c^m\at{e\algebra{A}^{(M)}}$ gives usual dual bases for
    $\algebra{A}_\Total^{(M_\Total)}$ and
    $\algebra{A}_\Wobs^{(M_\Wobs)}$.  Thus we get
    \eqref{prop:DualBasis_Eq1}.  Since $e$ is a morphism of
    coisotropic modules we have \ref{prop:DualBasis_1} and
    \ref{prop:DualBasis_2} by the definition of $e_n$ and $f_m$.  For
    $x \in \algebra{A}_\Wobs^{(M_\Wobs)}$ it holds that
    \begin{align*}
        b^n(\iota_{\algebra{A}^{(M)}}(x))
        =
        b^n \Bigl(
        \iota_{\algebra{A}^{(M)}}
        \Bigl( \sum_{m \in M_\Wobs} c_m x_m\Bigr)
        \Bigr)
        =
        b^n \Bigl(
        \sum_{m \in M_\Wobs} b_{\iota_M(m)} \iota_\algebra{A}(x_m)
        \Bigr),
    \end{align*}
    and therefore we get
    $b^{\iota_M(m)} \circ \iota_{\algebra{A}^{(M)}} =
    \iota_\algebra{A} \circ c^m$ by setting $n = \iota_M(m)$ and
    $b^n \circ \iota_{\algebra{A}^{(M)}} = 0$ for
    $n \in M_\Total\setminus \iota_M(M_\Wobs)$.  Since $e^n$ and $f^m$
    are defined as restrictions of $b^n$ and $c^m$ to
    $e\algebra{A}^{(M)}$ we get \ref{prop:DualBasis_3} and
    \ref{prop:DualBasis_4}.  If now
    $x \in \module{P}_\Null \simeq e_\Wobs(\algebra{A}^{(M)})_\Null
    \subseteq \algebra{A}_\Wobs^{(M_\Null)} +
    \algebra{A}_\Null^{(M_\Wobs)}$, then
    $c^m(x) \in \algebra{A}_\Null$ for
    $m \in M_\Wobs \setminus M_\Null$ and therefore
    \ref{prop:DualBasis_5} holds.  Let now such a dual basis in the
    above sense be given.  The map $M \to \module{P}$ of coisotropic
    index sets defined by $n \mapsto e_n$ and $m \mapsto f_m$ is a
    morphism of coisotropic index sets because of
    \ref{prop:DualBasis_1} and \ref{prop:DualBasis_2}.  By the
    universal property of free coisotropic modules we thus get an
    induced morphism $q \colon \algebra{A}^{(M)} \to \module{P}$.  We
    define $\ins \colon \module{P} \to \algebra{A}^{(M)}$ by
    \begin{equation*}
        \ins_\Total \colon x_\Total
        \mapsto
        \sum_{n \in M_\Total} b_n e^n(x_\Total)
        \quad
        \textrm{and}
        \quad
        \ins_\Wobs \colon x_\Wobs
        \mapsto
        \sum_{m \in M_\Wobs} c_m f^m (x_\Wobs).
    \end{equation*}
    The maps $\ins_\Total$ and $\ins_\Wobs$ are clearly module
    morphisms as the $e^m$ and $f^m$ are.  To see that these form a
    morphism of coisotropic modules observe that
    \begin{align*}
        (\ins_\Total \circ \iota_\module{P})(x_\Wobs)
        &=
        \sum_{n \in M_\Total}
        b_n e^n(\iota_{\module{P}}(x_\Wobs))
        =
        \sum_{m \in M_\Wobs}
        b_{\iota_M(m)} \iota_\algebra{A}(f^m (x_\Wobs))
        \\
        &=
        \iota_{\algebra{A}^{(M)}}
        \Bigl(\sum_{m \in M_\Wobs} c_{m} f^m(x_\Wobs) \Bigr)
        =
        (\iota_{\algebra{A}^{(M)}} \circ \ins_\Wobs)(x_\Wobs)
    \end{align*}
    for $x_\Wobs \in \module{P}_\Wobs$ by \ref{prop:DualBasis_3} and
    \ref{prop:DualBasis_4}.  Moreover, for $x_0 \in \module{P}_\Null$
    \begin{equation*}
        \ins_\Wobs(x_0)
        =
        \sum_{m \in M_\Wobs} c_m f^m(x_0)
        =
        \sum_{m \in M_\Null} c_m f^m(x_0)
        +
        \sum_{m \in M_\Wobs \setminus M_\Null} c_m f^m(x_0)
        \in (\algebra{A}^{(M)})_\Null
    \end{equation*}
    holds due to \ref{prop:DualBasis_5}.  Thus we have seen that
    $\ins$ is a morphism of coisotropic modules.  We now show
    $q \circ \ins = \id_\module{P}$: For
    $x_\Total \in \module{P}_\Total$ we have
    \begin{equation*}
        q_\Total (\ins_\Total(x_\Total))
        =
        q_\Total\Bigl(\sum_{n \in M_\Total} b_n e^n(x_\Total)\Bigr)
        =
        \sum_{n \in M_\Total} e_n e^n(x_\Total)
        =
        x_\Total
    \end{equation*}
    by assumption and likewise for $x_\Wobs \in \module{P}_\Wobs$ we
    have
    \begin{equation*}
        q_\Wobs (\ins_\Wobs(x_\Wobs))
        =
        q_\Wobs\Bigl(\sum_{m \in M_\Wobs} c_m f^m(x_\Wobs)\Bigr)
        =
        \sum_{m \in M_\Wobs} f^m(x_\Wobs)
        =
        x_\Wobs.
    \end{equation*}
    Thus the coisotropic endomorphism
    $e \coloneqq \ins \circ q \in \End_\algebra{A}(\algebra{A}^{(M)})$
    is an idempotent and $\module{P} \simeq e\algebra{A}^{(M)}$ via
    the maps $\ins$ and $q\at{e\algebra{A}^{(M)}}$.  Hence
    $\module{P}$ is regular projective.
\end{proof}
\begin{remark}
    We have seen in \autoref{Example:Free_Modules},
    \ref{item:AlgebraIsFreeModule} that duals of finitely generated
    $\SetTripleInj$-free coisotropic modules are not in general
    regular projective.  Therefore, duals of finitely generated
    regular projective coisotropic modules may fail to be regular
    projective.
\end{remark}

The notion of regular projectivity is compatible with the reduction
functor
$\red \colon \cCoisoModTriple{\algebra{A}} \to
\cModules{\algebra{A}_\red}$:
\begin{proposition}
    \label{prop:ProjectiveReduction}%
    Let $\module{P} \in \cCoisoModTriple{\algebra{A}}$ be regular
    projective with generating set $M \in \SetTripleInj$.  Then
    $\module{P}_\red \in \cModules{\algebra{A}_\red}$ is projective
    with generating set $M_\Wobs \setminus M_\Null$.
\end{proposition}
\begin{proof}
    This is a consequence of
    $\bigl(\algebra{A}^{(M)}\bigr)_\red =
    \bigl(\algebra{A}_\red\bigr)^{(M_\Wobs \setminus M_\Null)}$ and
    \autoref{thm:RegularProjective}, \ref{thm:RegularProjective_4}.
\end{proof}
\begin{example}
    \label{ex:ClassicalProjModule}%
    Every classical projective module $\module{E}$ over a classical
    algebra $\algebra{A}$ can be considered as a regular projective
    coisotropic module $(\module{E},\module{E},0)$ over the
    coisotropic algebra $(\algebra{A},\algebra{A},0)$.  Note that
    in general there will be more regular projective modules over
    the coisotropic algebra $(\algebra{A},\algebra{A},0)$.
\end{example}
\begin{remark}
    As noted in \autoref{remark:Coisotropic_Sets}, the concepts of
    coisotropic index sets $\SetTriple$ and coisotropic sets
    $\CoisoSetTriple$ as in \cite{dippell.esposito.waldmann:2020a}
    do not agree.
    However, it can be shown that regular projective coisotropic
    modules can equivalently be described using the corresponding
    notion of $\CoisoSetTripleInj$-free coisotropic modules:
    Again they can be viewed as direct summands or images of
    idempotents on such free coisotropic modules and also a version
    of a dual basis exists.
    Nevertheless, it turns out that working with $\SetTripleInj$
    is more convenient later on when we construct a dual basis
    using geometric data.
\end{remark}

%
% Vector Bundles, Distributions and Quotients
%

\section{Vector Bundles, Distributions and Quotients}

\autoref{thm:RegularProjective} shows that for a regular projective
module $\module{E}$ over a coisotropic algebra $\algebra{A}$ the
components $\module{E}_\Total$ and $\module{E}_\Wobs$ are projective
$\algebra{A}_\Total$- and $\algebra{A}_\Wobs$-modules, respectively.
Thus, in order to describe the regular projective modules over
$\Cinfty(M,C,D) = (\Cinfty(M), \Cinfty_D(M), \algebra{J}_C)$ as in
\autoref{ex:CoisotropicAlgebras}, \ref{ex:CoisotropicAlgebras_1}, it
makes sense to first focus on the $\Total$- and $\Wobs$-components
separately.

It is well-known that projective modules over the smooth functions on
a manifold are essentially sections of a vector bundle over said
manifold, see
e.g.~\cite[Thm.~11.32~and~Remark~thereafter]{nestruev:2003a}:
\begin{remark}[Serre-Swan Theorem]
    \label{Theorem:Serre_Swan}%
    Let $E \to N$ be a vector bundle over a smooth manifold.
    \begin{remarklist}
    \item The sections $\Secinfty(E)$ are a finitely generated
        projective module over $\Cinfty(N)$.
    \item If $N$ is connected, then taking sections yields an
        equivalence of categories
        \begin{equation}
            \Secinfty \colon \Vect(N) \to \Proj(\Cinfty(N)) ,
        \end{equation}
        where $\Vect(N)$ is the category of vector bundles over $N$
        together with vector bundle morphisms over the identity and
        $\Proj(\Cinfty(N))$ are the finitely generated projective
        modules over $\Cinfty(N)$ together with module morphisms.
    \end{remarklist}
\end{remark}

This already describes the $\Total$-components of the regular
projective modules of interest.  To investigate the
$\Wobs$-components, we first restrict ourselves to the submanifold
$C \subseteq M$: Given an integrable distribution $D \subseteq TC$ on
a smooth manifold $C$ we want to understand the projective modules
over $\Cinfty_D(C)$, the functions constant along the leaves of $D$.

%
% $D$-Connections and Quotients of Vector Bundles
%

\subsection{$D$-Connections and Quotients of Vector Bundles}

If $D \subseteq TC$ is simple, i.e. the leaf space $C \slash D$
carries the structure of a smooth manifold such that the projection
$\pi \colon C \to C \slash D$ is a surjective submersion, then
\begin{equation}
    \pi^* \colon \Cinfty(C \slash D) \to \Cinfty_D(C)
\end{equation}
is an isomorphism of algebras. Therefore, instead of understanding
modules over $\Cinfty_D(C)$ we can as well consider modules over
$\Cinfty(C \slash D)$.

It turns out that---similarly to $\Cinfty_D(C)$ being the functions
constant along the leaves---the projective $\Cinfty_D(C)$-modules are
sections on vector bundles over $C$ which are parallel with respect to
some covariant derivative.  However, only the part of the covariant
derivative in direction of $D$ plays a role.  This leads to the
following definition:
\begin{definition}[$D$-Connection]
    Let $D \subseteq TC$ be a smooth (possibly singular) distribution
    and let $E \to C$ be a vector bundle.  A map
    $\nabla\colon D \times \Secinfty(E) \to E$ is called
    \emph{$D$-connection on $E$} if it is the restriction of a
    covariant derivative $\nabla \colon TC \times \Secinfty(E) \to E$.
\end{definition}
Similar concepts, sometimes called \emph{partial connections}, are
used in \cite[Ch.~6]{bott:1972a}, \cite{kamber.tondeur:1975a},
\cite[Def.~2.2]{jotz.ortiz:2014a}.
A generalization of this notion is used in the context of Lie
algebroids, see e.g. \cite{fernandes:2002a}. Later, requiring
the $D$-connections to be restrictions of covariant derivatives on
$TC$ will ensure some smoothness also in direction transverse to the
leaves.  Indeed, in the case of a locally constant rank the following
statement is easy to prove using projections onto the \emph{subbundle}
$D$ (on connected components of $C$):
\begin{lemma}
    \label{Lemma:Covariant_Derivative_Restriction_Locally_Regular}%
    Let $D\subseteq TC$ be a smooth distribution with locally constant
    rank.  Let $\nabla\colon D \times \Secinfty(E) \to E$ be a
    $\field{R}$-bilinear map such that:
    \begin{lemmalist}
    \item We have $\nabla_{v_p} s \in E_p$ for $v_p \in D$.
    \item There is the following Leibniz rule in the second argument,
        i.e.
        \begin{equation}
            \nabla_{v_p} (fs) = f(p) \nabla_{v_p} s + v_p(f) s(p)
	\end{equation}
        for $v_p \in D$, $f \in \Cinfty(C)$, and $s\in\Secinfty(E)$.
    \item The induced map
        \begin{equation}
            \nabla\colon \Secinfty(D) \times \Secinfty(E) \to \Secinfty(E)
            \quad
            \textrm{with}
            \quad
            \nabla_X s \at{p} \coloneqq \nabla_{X(p)} s
	\end{equation}
        indeed maps into $\Secinfty(E)$.
    \end{lemmalist}
    Then $\nabla$ is a $D$-connection.
\end{lemma}

In the singular setting there are maps fulfilling the requirements of
\autoref{Lemma:Covariant_Derivative_Restriction_Locally_Regular} which
are not restrictions of a covariant derivative on $TC$:
\begin{example}
    \label{Ex:Covariant_not_restriction}%
    Consider the (integrable) smooth singular distribution
    $D \subseteq T\field{R}$ defined by
    \begin{equation}
	D_p \coloneqq
        \begin{cases}
            T_p \field{R} & \textrm{if } p > 0 \\
            \lbrace 0_p \rbrace & \textrm{if } p \leq 0.
        \end{cases}
    \end{equation}
    We define
    $\nabla \colon D \times \Secinfty(T\field{R}) \to T \field{R}$ by
    \begin{equation}
	\nabla_{v \partial_x\at{p}}
	\bigl(g \partial_x\bigr)
	\coloneqq
	v \cdot g^\prime(p) \cdot \partial_x\at[\Big]{p}
	+
        v \cdot \frac{1}{p} \cdot g(p) \cdot \partial_x\at[\Big]{p}
    \end{equation}
    for $p>0$, $v\in\field{R}$, and $g \in \Cinfty(\field{R})$.  For
    $p \leq 0$ the only element of $D_p$ is $0_p$ and we set
    \begin{equation}
        \nabla_{0_p} \bigl(g \partial_x\bigr) \coloneqq 0_p .
    \end{equation}
    This fulfills the requirements of
    \autoref{Lemma:Covariant_Derivative_Restriction_Locally_Regular}.
    Indeed, it maps smooth sections to smooth sections because for
    $f \in \Cinfty(\field{R})$ we have $f \partial_x \in \Secinfty(D)$
    iff $f(p) = 0$ for all $p \leq 0$.  But clearly $\nabla$ can not
    be the restriction of a covariant derivative $\widetilde{\nabla}$
    in direction $T\field{R}$ because we would have
    \begin{equation}
        \widetilde{\nabla}_{\partial_x\at{p}} \bigl( \partial_x \bigr)
        =
        \frac{1}{p} \partial_x\at[\Big]{p}
    \end{equation}
    for all $p>0$.
\end{example}

In order to show projectivity of parallel sections on a vector bundle
over $C$ we want to identify them with sections over a suitable vector
bundle over $C \slash D$.  Denote by $\Parallel_{\gamma, p \to q}$ the
parallel transport along a curve $\gamma$ from $p$ to $q$ inside a
fixed leaf with respect to a given $D$-connection. We need to assume
that the parallel transport $\Parallel$ is path-independent and call
such $D$-connections \emph{holonomy-free}. Then we define the
following map on the fibered product
$E \times_\pi C = \lbrace (v_p,q) \subseteq E \times C \mid \pi(p) =
\pi(q) \rbrace$:
\begin{definition}[$\nabla$-Transport map]
    Let $D \subseteq TC$ be simple with projection
    $\pi \colon C \to C \slash D$ and $\nabla$ a holonomy-free
    $D$-connection on a vector bundle $E \to C$.  Then we call
    \begin{equation}
        \Transport^\nabla \colon E \times_\pi C \to E,
        \quad
        (v_p, q) \mapsto \Parallel_{\gamma, p \to q}(v_p),
    \end{equation}
    where $\gamma$ is any smooth curve inside the leaf
    $\pi^{-1}(\lbrace \pi (p) \rbrace) = \pi^{-1}(\lbrace \pi (q)
    \rbrace)$ connecting $p$ and $q$ the
    \emph{$\nabla$-transport map}.
\end{definition}
Note that the parallel transport along leaves does not depend on the
extension of $\nabla$ to $TC$.  Some properties of this map are
summarized in the next statement:
\begin{lemma}
    \label{Lemma:Properties_Transport}%
    Let $\nabla$ be a holonomy-free $D$-connection with transport map
    $\Transport^\nabla$.
    \begin{lemmalist}
    \item For $(v_p, q) \in E \times_\pi C$ one has
        $\Transport^\nabla(v_p,q) \in E_q$.
    \item The transport map $\Transport^\nabla$ is linear in the first
        argument, i.e. one has
        $\Transport^\nabla(\lambda v_p + \mu w_p, q) = \lambda
        \Transport^\nabla(v_p,q) + \mu \Transport^\nabla(w_p,q)$ for
        $(v_p,q),(w_p,q) \in E \times_\pi C$.
    \item For $v_p \in E$ one has $\Transport^\nabla(v_p,p) = v_p$.
    \item\label{Item:Transport_Composition} For
        $(v_p, q) \in E \times_\pi C$ and
        $r \in \pi^{-1}(\pi(\lbrace p \rbrace )) =
        \pi^{-1}(\pi(\lbrace q \rbrace ))$ one has
        $\Transport^\nabla\bigl( \Transport^\nabla(v_p,q),r\bigr) =
        \Transport^\nabla(v_p, r)$.
    \item The transport map $\Transport^\nabla$ is smooth.
    \end{lemmalist}
\end{lemma}
\begin{proof}
    The algebraic properties are clear.  For smoothness note that by
    standard arguments from the theory of ordinary differential
    equations, for any $p \in C$ there exists an open neighborhood
    $U \subseteq C$ such that $\Transport^\nabla$ is smooth on
    $E\at{U} \times_\pi U$.  Using the property
    \ref{Item:Transport_Composition} we can iterate this statement to
    see that for any $p,q \in C$ there exist open neighborhoods of
    $U_p$ of $p$ and $U_q$ of $q$ such that $\Transport^\nabla$ is
    smooth on $E\at{U_p} \times_\pi U_q$.
\end{proof}

In \cite[Def.~2.1.1]{mackenzie:2005a} this is called a \emph{linear
  action of $R(\pi)$ on $E$} where $R(\pi) \subseteq C \times C$ is
the equivalence relation induced by $\pi$.  The $\nabla$-transport map
induces an equivalence relation on the vector bundle:
\begin{definition}[$\nabla$-Equivalence]
    Let $D \subseteq TC$ be simple with projection
    $\pi \colon C \to C \slash D$ and let $\nabla$ be a holonomy-free
    $D$-connection.  Then for $v_p,w_q \in E$ we set
    \begin{equation}
	v_p \sim_{\nabla} w_q \iff \pi(p) = \pi(q)
        \textrm{ and }
        \Transport^\nabla(v_p, q) = w_q .
    \end{equation}
\end{definition}
Its set of equivalence classes $E \slash \mathord{\sim}_\nabla$ offers
a good candidate for a projection onto $C \slash D$ which could turn
it into a vector bundle over $C \slash D$: We simply map an
equivalence class $[v_p]$ to $\pi(p)$.  Indeed, we can equip
$E \slash \mathord{\sim}_\nabla$ with the structure of a vector bundle
over $C \slash D$:
\begin{proposition}
    \label{Prop:Quotient_Vector_Bundles}%
    Let $D \subseteq TC$ be simple with projection
    $\pi \colon C \to C \slash D$ and let $\nabla$ be a holonomy-free
    $D$-connection on a vector bundle $E \to C$.  Then there exists a
    unique vector bundle structure on
    \begin{equation}
	\pr_\nabla \colon E \slash \mathord{\sim}_\nabla \to C \slash
        D,
        \quad
        [v_p] \mapsto \pi(p)
    \end{equation}
    such that the quotient map
    \begin{equation}
	\pi_\nabla \colon E \to E \slash \mathord{\sim}_\nabla,
        \quad
        v_p \mapsto [v_p]
    \end{equation}
    is both a submersion and a vector bundle morphism over $\pi$.
    Moreover, there exists an isomorphism
    \begin{equation}
        \label{Eq:Isomorphism_Quotient_VB_Pullback}
	\Xi \colon E
        \to
        \pi^\sharp \bigl( E \slash \mathord{\sim}_\nabla \bigr),
        \quad
	v_p \mapsto (p, [v_p])
    \end{equation}
    of vector bundles over the identity $\id_C$ and this isomorphism
    fulfills
    \begin{equation}
	\Parallel_{\gamma, p \to q}(v_p) = \Xi^{-1}(q,[v_p])
    \end{equation}
    for $v_p \in E$ and $\pi(p) = \pi(q)$.
\end{proposition}
\begin{proof}
    Using \autoref{Lemma:Properties_Transport},
    \cite[Prop.~4.1]{higgins.mackenzie:1990a}
    gives the result.
\end{proof}

A similar quotient is considered in the context of infinitesimal ideal
systems in \cite[Thm.~3.7]{zambon:2008a} and in
\cite[Ch.~6]{jotz.ortiz:2014a}: If we turn a given vector bundle
$E \to C$ into a Lie algebroid with trivial anchor and trivial Lie
bracket and use $(D, C \times \lbrace 0 \rbrace, \nabla)$ as
infinitesimal ideal system, then the constructions of
\cite[Cor.~6.3]{jotz.ortiz:2014a} and
\autoref{Prop:Quotient_Vector_Bundles} coincide.
\begin{remark}
    Other examples of linear actions of $R(\pi)$ on $E$, i.e. maps
    $\Transport$ with the properties of
    \autoref{Lemma:Properties_Transport}, can be found in the context
    of group actions: Let $\pi \colon C \to C \slash G$ be a left
    principal fiber bundle with structure group $G$ and principal
    action $\acts^C$. Moreover, let $E \to C$ be a $G$-equivariant
    vector bundle with left action $\acts^E$.  Then using the ratio
    map $\mathrm{r} \colon C \times_\pi C \to G$ defined by
    $\mathrm{r}(p,q) \acts^C p = q$ we can set
    \begin{equation}
        \Transport(v_p, q) \coloneqq \mathrm{r}(p,q) \acts^E v_p .
    \end{equation}
    With $v_p \sim_\Transport w_q \iff \Transport(v_p,q) = w_q$ we
    obtain
    $E \slash \mathord{\sim}_\Transport = E \slash G \to C \slash G$
    as quotient vector bundle.
\end{remark}

Ultimately, we are interested in sections which are parallel in the
direction of a distribution:
\begin{definition}[$\nabla$-Parallel sections]
    Let $D \subseteq TC$ be a distribution and let $\nabla$ be a
    $D$-connection on a vector bundle $E \to C$.  Then the
    \emph{$\nabla$-parallel sections} of $E$ are
    \begin{equation}
	\Secinfty_D(E,\nabla)
	\coloneqq
        \lbrace
        s \in \Secinfty(E)
	\mid
        \nabla_X s = 0 \quad \forall X \in \Secinfty(D)
        \rbrace.
    \end{equation}
\end{definition}
As to be expected this condition can be integrated to a condition
involving the parallel transport,
cf. \cite[Prop.~6.5]{jotz.ortiz:2014a}:
\begin{lemma}
    \label{Lemma:Parallel_Sections_Parallel_Transport}%
    Let $D \subseteq TC$ be an integrable distribution and let
    $\nabla$ be a $D$-connection on a vector bundle $E \to C$.  Then
    \begin{equation}
	\Secinfty_D(E, \nabla)
	=
        \lbrace
        s \in \Secinfty(E)
        \mid
        \Parallel_{\gamma, a \to b}(s(\gamma(a)))
	=
        s(\gamma(b))
        \textrm{ for all smooth curves }
        \gamma
        \textrm{ inside a leaf of }
        D
	\rbrace.
    \end{equation}
\end{lemma}
With \eqref{Eq:Isomorphism_Quotient_VB_Pullback} at hand, we see that
the $\nabla$-parallel sections are nothing else but the sections over
the quotient vector bundle:
\begin{lemma}
    \label{Lemma:Isomorphism_Invariant_Sections_Quotient_Sections}%
    Let $D \subseteq TC$ be simple with projection
    $\pi \colon C \to C \slash D$ and let $\nabla$ be a holonomy-free
    $D$-connection on a vector bundle $E \to C$.  Then with $\Xi$ as
    in \eqref{Eq:Isomorphism_Quotient_VB_Pullback}
    \begin{equation}
	\Secinfty(E \slash \mathord{\sim}_\nabla) \to \Secinfty_D(E, \nabla),
	\quad
	r \mapsto \bigl( C \ni p \mapsto
	\Xi^{-1}\bigl( p, r(\pi(p))\bigr) \in E \bigr)
    \end{equation}
    is an isomorphism of modules over the isomorphism
    $\pi^* \colon \Cinfty(C \slash D) \to \Cinfty_D(C)$ of algebras.
    Its inverse is given by
    \begin{equation}
        \label{Eq:Isomorphism_Invariant_Sections_Inverse}
	\Secinfty_D(E, \nabla)
        \to
        \Secinfty(E \slash \mathord{\sim}_\nabla),
	\quad
	s
        \mapsto
        \bigl(
        C \slash D \ni q \mapsto [s(\pi^{-1}(q))]
	\in E \slash \mathord{\sim}_\nabla
        \bigr),
    \end{equation}
    where $\pi^{-1}(q)$ is an arbitrary element in
    $\pi^{-1}(\lbrace q \rbrace)$.
\end{lemma}

%
% The Module $\Secinfty_D(E, \nabla)$
%

\subsection{The Module of Invariant Sections}
\label{subsec:ModuleSecInftyENabla}

Using \autoref{Theorem:Serre_Swan} and applying
\autoref{Lemma:Isomorphism_Invariant_Sections_Quotient_Sections} to
both the sections and the dual sections we finally obtain a dual basis
for invariant sections:
\begin{proposition}
    \label{Prop:Dual_Basis_Invariant_Sections}%
    Let $D \subseteq TC$ be simple with projection
    $\pi \colon C \to C \slash D$ and let $\nabla$ be a holonomy-free
    $D$-connection on a vector bundle $E \to C$.  Then there exists a
    finite index set $I$ and $\nabla$-parallel sections
    $e_i \in \Secinfty_D(E , \nabla)$ and
    $e^i \in \Secinfty_D(E^*, \nabla^*)$ for $i \in I$ such that
    \begin{equation}
	v_p = \sum_{i \in I} e_i(p) e^i(v_p)
    \end{equation}
    for all $v_p \in E$.
\end{proposition}
\begin{proof}
    Let $I$ be a finite index set and let
    $f_i \in \Secinfty(E \slash \mathord{\sim}_\nabla)$ and
    $f^i \in \Secinfty\bigl( (E \slash \mathord{\sim}_\nabla)^*\bigr)$
    for $i \in I$ such that
    \begin{equation*}
	u_x = \sum_{i \in I} f_i(x) f^i(u_x)
    \end{equation*}
    for all $u_x \in E \slash \mathord{\sim}_\nabla$ at
    $x \in C \slash D$.
    \autoref{Lemma:Isomorphism_Invariant_Sections_Quotient_Sections}
    gives $e_i \in \Secinfty_D(E,\nabla)$ with
    $[e_i(p)] = f_i(\pi(p))$ for all $p \in C$.  Setting
    $e^i(v_p) \coloneqq f^i([v_p])$ gives
    $e^i \in \Secinfty_D(E^*, \nabla^*)$ with
    \begin{equation*}
	\Bigl[ \sum_{i \in I} e_i(p) e^i(v_p) \Bigr]
        =
        \sum_{i \in I} [e_i(p)] f^i([v_p]) = [v_p]
    \end{equation*}
    for all $v_p \in E$. But then also
    \begin{equation*}
	\sum_{i \in I} e_i(p) e^i(v_p) = v_p
    \end{equation*}
    because
    $\pi_\nabla \colon E_p \to \bigl( E \slash \mathord{\sim}_\nabla
    \bigr)_{\pi(p)}$ is bijective.
\end{proof}

For singular $D$ we are not guaranteed that taking parallel sections
gives us a finitely generated projective $\Cinfty_D(C)$-module.
Nevertheless, in one dimension we have the following positive result:
\begin{example}
    \label{Ex:Singular_Projective_1D}%
    Let $C = \field{R}$ and $E \to \field{R}$ be a vector bundle.  For
    dimensional reasons the curvature of any covariant derivative on
    $E$ vanishes and therefore the parallel transport is always
    path-independent.  In particular, if $D \subseteq TC$ is some
    integrable distribution and $\nabla$ is a $D$-connection on $E$,
    then it has an extension to a holonomy-free $TC$-connection
    $\hat{\nabla}$.  Therefore by
    \autoref{Prop:Dual_Basis_Invariant_Sections} there exists a finite
    dual basis $e_i \in \Secinfty_{TC}(E,\hat{\nabla})$ and
    $e^i \in \Secinfty_{TC}(E^*, \hat{\nabla}^*)$ for $i \in I$ with
    \begin{equation}
        v_p = \sum_{i \in I} e_i(p) e^i(v_p)
    \end{equation}
    for all $v_p \in E$. In particular $e_i \in \Secinfty_D(E,\nabla)$
    and
    $e^i \in \Hom_{\Cinfty_D(C)}(\Secinfty_D(E,\nabla), \Cinfty_D(C))$
    yield a dual basis of $\Secinfty_D(E, \nabla)$. Therefore
    $\Secinfty_D(E, \nabla)$ is a finitely generated projective
    $\Cinfty_D(C)$-module.
\end{example}

Our condition that a $D$-connection arises from a covariant derivative
defined on all tangent vectors becomes crucial in the singular
case. Without this assumption, the module $\Secinfty_D(E, \nabla)$
might fail to be finitely generated and projective:
\begin{example}
    We reconsider the situation from
    \autoref{Ex:Covariant_not_restriction}, i.e.
    \begin{equation}
        D_p \coloneqq
        \begin{cases}
            T_p \field{R} & \textrm{if } p > 0 \\
            \lbrace 0_p \rbrace & \textrm{if } p \leq 0,
        \end{cases}
    \end{equation}
    and the $\nabla$ which was not the restriction of a covariant
    derivative. Then the parallel sections are given by
    \begin{equation}
        A
        \coloneqq
        \lbrace
        f \in \Cinfty(\field{R})
        \mid
        f(x) = 0
        \textrm{ for }
        x > 0
        \rbrace.
    \end{equation}
    This $\Cinfty_D(C)$-module $A$ is neither finitely generated nor
    projective.  In \cite[Prop.~5.3]{drager.lee.park.richardson:2012a}
    it is shown that $A$ is not a finitely generated
    $\Cinfty(\field{R})$-module. Thus it is not a finitely generated
    $\Cinfty_D(C)$-module either.  To see that it is not projective,
    assume that there exists a dual basis $e_i \in A$,
    $e^i \in \Hom_{\Cinfty_D(C)}(A, \Cinfty_D(C))$ for some (possibly
    infinite) index set $I$ such that for fixed $s \in A$ only
    finitely many $e^i(s) \neq 0$ and
    \begin{equation}
        s = \sum_{i \in I} e_i e^i(s).
    \end{equation}
    If $r \in A$ and $r(x) = 0$ for some $x < 0$, then we can
    choose some bump function $\zeta \in A \subseteq \Cinfty_D(C)$
    with $\zeta(x) = 1$.  It follows that
    \begin{equation}
        e^i(r)\at{x}
        =
        \zeta(x) e^i(r)\at{x} = e^i(\zeta r)\at{x}
        =
        r(x) e^i(\zeta)\at{x} = 0 .
    \end{equation}
    Now for $\hat{s}(x) \coloneqq \chi(-x)$ with $\chi$ as above we
    have $\hat{s} \in A$ and therefore only finitely many of the
    $e^i(\hat{s}) \neq 0$.  We denote the indices by
    $j_1, \ldots, j_n$. If $s \in A$ is arbitrary, then for fixed
    $x < 0$
    \begin{equation}
        s^\prime
        \coloneqq
        s - \frac{s(x)}{\hat{s}(x)} \zeta \hat{s} \in A
    \end{equation}
    fulfills $s^\prime(x) = 0$ and therefore $e^i(s^\prime)\at{x} = 0$
    for all $i \in I$.  This shows that because
    $\frac{s(x)}{\hat{s}(x)} \zeta \in A\subseteq \Cinfty_D(C)$ we
    have
    \begin{equation}
        e^i(s)\at{x}
        =
        \frac{s(x)}{\hat{s}(x)} \zeta(x) e^i(\hat{s})\at{x}
        =
        \frac{s(x)}{\hat{s}(x)} e^i(\hat{s})\at{x} .
    \end{equation}
    In particular, $e^i(s)\at{x} = 0$ for all
    $i \notin \lbrace j_1, \ldots, j_n \rbrace$, all $x < 0$ and all
    $s \in A$. As $e_i\at{y} = 0$ for all $y \geq 0$ we have
    $e_i e^i(s) = 0$ for all
    $i \notin \lbrace j_1, \ldots, j_n \rbrace$ and all $s \in A$
    which implies
    \begin{equation}
        s = \sum_{k = 1}^n e_{j_k} e^{j_k}(s) .
    \end{equation}
    In particular, $A$ would be finitely generated in contradiction to
    what we saw before. Therefore $A$ can not be projective.
\end{example}

The situation in the one-dimensional case was very special as we could
extend any $D$-connection to a covariant derivative with globally
path-independent parallel transport.  The next example considers a
situation in which this is not the case:
\begin{example}
    \label{Ex:Singular_Projective_2D}%
    Let $C = \field{R}^2$, let $E \to C$ be a vector bundle and let
    $D\subseteq TC$ be the integrable distribution defined by
    \begin{equation}
	D_p \coloneqq
        \begin{cases}
            T_p \field{R}^2 & \textrm{if } p_x > 0 \\
            \lbrace 0_p \rbrace & \textrm{if } p_x \leq 0.
        \end{cases}
    \end{equation}
    Let $\nabla$ be a holonomy-free $D$-connection.  For $p,q \in C$
    we define a smooth path connecting $p$ and $q$ by
    \begin{equation}
	\gamma_{p,q}(t) \coloneqq (1-t)p + tq .
    \end{equation}
    Clearly $\gamma_{p,q}$ depends smoothly on $p$ and $q$ and we
    obtain a smooth map
    \begin{equation}
	\Transport^\prime \colon E \times C \to E,
        \quad
	(v_p, q) \coloneqq \Parallel^\nabla_{\gamma_{p,q}, 0 \to 1}(v_p) .
    \end{equation}
    Fixing some point $p_0$ in the non-trivial leaf
    $\field{R}^+ \times \field{R}$ of $D$, for example $p_0 = (1,0)$,
    and a basis $b_1, \ldots, b_n$ of $E_{p_0}$ with dual basis
    $b^1, \ldots, b^n$ we can define sections by setting
    \begin{equation}
	e_i(p) \coloneqq \Transport^\prime(b_i, p) \quad \text{and} \quad
	e^i(p) \coloneqq b^i\bigl(\Transport^\prime(v_p, p_0)\bigr)
    \end{equation}
    for $p \in C$. Then
    \begin{equation}
	\sum_{i=1}^n e_i(p) e^i(v_p)
	=
        \sum_{i=1}^n \Transport^\prime
        \Bigl(
	b^i\bigl(\Transport^\prime(v_p, p_0)\bigr) b_i, p
        \Bigr)
	=
        \Transport^\prime\bigl(\Transport^\prime(v_p, p_0), p\bigr)
	=
        v_p
    \end{equation}
    as $\gamma_{p,p_0}$ is the reverse curve to
    $\gamma_{p_0,p}$. Moreover, the $e_i,e^i$ really are a dual basis
    of $\Secinfty_D(E, \nabla)$ because $\Transport^\prime$ restricted
    to $\field{R}^+ \times \field{R}$ is just $\Transport^\nabla$.
\end{example}

%
% Projective Coisotropic $\Cinfty(M,C,D)$-Modules
%

\section{Projective Coisotropic $\Cinfty(M,C,D)$-Modules}

Now that we have described the finitely generated projective modules
over $\Cinfty_D(C)$ we again look at the whole coisotropic algebra
$\Cinfty(M,C,D)$ and want to find a description of its regular
projective coisotropic modules.  There are two main steps we have to
take: First of all, we need to extend from our submanifold
$C \subseteq M$ to $M$ in order to pass from $\Cinfty_D(C)$-modules to
$\Cinfty_D(M)$-modules.  As a second step we need to incorporate the
relations between the $\Total$-, the $\Wobs$- and the
$\Null$-components which become evident in \autoref{prop:DualBasis}.

%
% Regular Projective Coisotropic $\Cinfty(M,C,D)$-Modules
%

\subsection{Regular Projective Coisotropic $\Cinfty(M,C,D)$-Modules}

As we have seen in the previous section we need to look at sections of
vector bundles but also need the additional data of a covariant
derivative. In the following we will use $\iota^\sharp$ to denote the
restriction of vector bundles on $M$ as well as their sections to the
submanifold $C$.
\begin{proposition}[Triples of invariant sections]
    \label{Def:Triples_Inv_Sections}%
    Let $\iota \colon C \to M$ be a submanifold and $D \subseteq TC$
    an integrable distribution. Let $E_\Total \to M$ be a vector
    bundle, $E_\Wobs \subseteq E_\Total$ and
    $E_\Null \subseteq \iota^\sharp E_\Wobs \to C$ vector subbundles,
    and $\nabla$ a $D$-connection on $\iota^\sharp E_\Wobs$.  Then for
    $\VBQuad{E}\coloneqq (E_\Total, E_\Wobs,E_\Null, \nabla)$ the
    $\Cinfty(M,C,D)$-module $\Secinfty(\VBQuad{E})$ defined by
    \begin{equation}
        \begin{split}
            \Secinfty(\VBQuad{E})_\Total
            &\coloneqq \Secinfty(E_\Total)
            \\
            \Secinfty(\VBQuad{E})_\Wobs
            &\coloneqq
            \lbrace
            s \in \Secinfty(E_\Wobs)
            \mid
            \nabla_X \iota^\sharp s = 0 \ \forall X \in \Secinfty(D)
            \rbrace
            \\
            \Secinfty(\VBQuad{E})_\Null
            &\coloneqq
            \lbrace
            s \in \Secinfty(E_\Wobs)
            \mid
            \iota^\sharp s \in \Secinfty(E_\Null)
            \textrm{ and }
            \nabla_X \iota^\sharp s = 0
            \ \forall X \in \Secinfty(D)
            \rbrace
        \end{split}
    \end{equation}
    together with $\iota_{\Secinfty(\VBQuad{E})}$ being the inclusion
    of sections $\Secinfty(E_\Wobs) \to \Secinfty(E_\Total)$ is called
    the \emph{triple of invariant sections} subject to $\VBQuad{E}$.
\end{proposition}

\begin{remark}
    Triples of the form as in \autoref{Def:Triples_Inv_Sections} are
    not the only possible way to define $\Cinfty(M,C,D)$-modules using
    sections of vector bundles.  For example, it would suffice to
    specify $E_\Wobs$ on the submanifold $C$ and $\nabla$ for sections
    of $E_\Wobs / E_\Null$: More precisely, given a vector bundle
    $E_\Total \to M$ and subbundles
    $E_\Null \subseteq E_\Wobs \subseteq \iota^\sharp E_\Total \to C$
    we could consider the coisotropic module $\module{E}$ defined by
    \begin{equation}
        \begin{split}
            \module{E}_\Total
            &\coloneqq
            \Secinfty(E_\Total)
            \\
            \module{E}_\Wobs
            &\coloneqq
            \lbrace
            s \in \Secinfty(E_\Total)
            \mid
            \iota^\sharp s \in \Secinfty(E_\Wobs)
            \textrm{ and }
            \nabla_X [\iota^\sharp s] = 0
            \ \forall X \in \Secinfty(D)
            \rbrace
            \\
            \module{E}_\Null
            &\coloneqq
            \lbrace
            s \in \Secinfty(E_\Total)
            \mid
            \iota^\sharp s \in \Secinfty(E_\Null)
            \rbrace.
        \end{split}
    \end{equation}
    However, triples of this type will not fit into the description of
    regular projective $\Cinfty(M,C,D)$-modules obtained in
    \autoref{Thm:Coisotropic_Serre_Swan}.  Roughly speaking, the
    reason for this is the constancy of the rank of $e_\Wobs$ for
    an idempotent $e^2 = e$ on a free coisotropic
    $\Cinfty(M,C,D)$-module with $\image e = \module{E}$.
\end{remark}

Under certain additional assumptions we can prove that triples of the
form as in \autoref{Def:Triples_Inv_Sections} are regular projective:
\begin{proposition}
    \label{Prop:Set3Inj_projective_sections}%
    Let $\iota \colon C \to M$ be a closed submanifold and let
    $D \subseteq TC$ be simple with projection
    $\pi \colon C \to C \slash D$. Moreover, let
    $\VBQuad{E} \coloneqq (E_\Total, E_\Wobs, E_\Null, \nabla)$ be as
    in \autoref{Def:Triples_Inv_Sections} with the additional
    assumptions that
    \begin{propositionlist}
    \item $E_\Null$ is closed under $\nabla$, i.e.
        $\nabla_X s \in \Secinfty(E_\Null)$ for all
        $X\in\Secinfty(D), s \in \Secinfty(E_\Null)$,
    \item $\nabla$ is holonomy-free.
    \end{propositionlist}
    Then $\Secinfty(\VBQuad{E})$ is a finitely generated regular
    projective coisotropic $\Cinfty(M,C,D)$-module.
\end{proposition}

Before we prove projectivity, we will need to extend sections of
$\iota^\sharp E_\Wobs$ to sections on $E_\Wobs$.  This can be done by
locally writing a vector bundle $E$ over $M$ as the pull-back vector
bundle of the restriction $\iota^\sharp E$ to $C$:
\begin{lemma}
    \label{Lemma:Vector_Bundle_Isomorphism_Theta}%
    Let $\iota \colon C \to M$ be a submanifold and $E \to M$ a vector
    bundle.  Then on an open neighborhood $U \supseteq \iota(C)$ there
    exists a surjective submersion $t \colon U \to C$ with
    $t \circ \iota = \id_C$ and a vector bundle isomorphism
    \begin{equation}
	\Theta \colon E \at{U} \to t^\sharp (\iota^\sharp E)
    \end{equation}
    over the identity with
    $t^\sharp \circ \Theta \circ \iota^\sharp = \id_{\iota^\sharp E}$
    for the projections $\iota^\sharp \colon \iota^\sharp E \to E$ and
    $t^\sharp \colon t^\sharp(\iota^\sharp E) \to \iota^\sharp E$.
\end{lemma}
\begin{proof}
    We choose a vector bundle $\pr_F \colon F \to C$, an open
    neighborhood $U \subseteq M$ of $\iota(C) \subseteq M$ and a
    diffeomorphism $\tau \colon F \to U$ with
    $\tau \circ \iota_F = \iota$ for the zero section
    $\iota_F \colon C \to F$. This can be done using a tubular
    neighborhood.  We can define
    $t \coloneqq \pr_F \circ \tau^{-1}$.  Now we choose an arbitrary
    covariant derivative $\nabla$ on $E\at{U} \to U$.  For $p \in U$
    we define
    \begin{equation*}
	\gamma_p \colon I \to U,
        \quad
        \gamma_p(s)
        \coloneqq
	\tau\bigl((1-s) \tau^{-1}(p)\bigr) .
    \end{equation*}
    We then have $\gamma_p(0) = p$ and $\gamma_p(1) = t(p)$.  We
    obtain a family of smooth curves which depends smoothly on
    $p \in U$. Setting
    \begin{equation*}
	\Theta(v_p)
        \coloneqq
        \bigl(p, \Parallel^\nabla_{\gamma_p, 0 \to 1}(v_p) \bigr),
    \end{equation*}
    we obtain a vector bundle morphism
    $\Theta \colon E \at{U} \to t^\sharp (\iota^\sharp E)$.  Its
    inverse is given by
    \begin{equation*}
	(p,v_{t(p)}) \mapsto \Parallel^\nabla_{\gamma_p, 1 \to 0}(v_{t(p)}) .
    \end{equation*}
\end{proof}

Moreover, we will need a decomposition of $\iota^\sharp E_\Wobs$ into
$E_\Null$ and a complementary subbundle $E_\Null^\perp$ which is also
closed under $\nabla$:
\begin{lemma}
    \label{Lemma:Complementary_Subbundle_Closed}%
    Let $C$ be a smooth manifold, $D \subseteq TC$ simple with
    projection $\pi \colon C \to C \slash D$ and $\nabla$ a
    holonomy-free $D$-connection on a vector bundle $E \to C$.  If
    $F \subseteq E$ is a vector subbundle such that
    \begin{equation}
	\nabla_X s \in \Secinfty(F)
    \end{equation}
    for all $s \in \Secinfty(F), X \in \Secinfty(D)$, then there
    exists a complementary vector subbundle $F^\perp$ such that
    \begin{equation}
	\nabla_X r \in \Secinfty(F^\perp)
    \end{equation}
    for all $r \in \Secinfty(F^\perp), X \in \Secinfty(D)$.
\end{lemma}
\begin{proof}
    By assumption the restriction $\nabla^F$ of $\nabla$ to $F$
    defines a $D$-connection on $F$.  Its parallel transport coincides
    with the parallel transport of elements of $F$ with respect to
    $\nabla$. In particular, the parallel transport with respect to
    $\nabla^F$ is path-independent along the leaves of $D$.
    Therefore, we can take the vector bundle quotient of $F$ with
    respect to $\nabla^F$ and obtain a vector subbundle
    $F \slash \mathord{\sim}_{\nabla^F} \subseteq E \slash
    \mathord{\sim}_\nabla$: Indeed, a suitable injective vector bundle
    morphism over the identity is given by
    \begin{equation*}
	F \slash \mathord{\sim}_{\nabla^F} \ni [v_p]
        \mapsto
        [v_p] \in E \slash \mathord{\sim}_\nabla .
    \end{equation*}
    The fact that we have a vector subbundle
    $F \slash \mathord{\sim}_{\nabla^F} \subseteq E \slash
    \mathord{\sim}_\nabla$ implies that there exists a complementary
    subbundle $Q \subseteq E \slash \mathord{\sim}_\nabla$, i.e.
    \begin{equation*}
	\bigl(F \slash \mathord{\sim}_{\nabla^F}\bigr) \oplus Q
        =
        E \slash \mathord{\sim}_\nabla .
    \end{equation*}
    Pulling $Q$ back gives a subbundle
    $\pi^\sharp Q \subseteq \pi^\sharp \bigl(E \slash
    \mathord{\sim}_\nabla\bigr)$.  Using the isomorphism
    \begin{equation*}
	\Phi \colon
        E \to \pi^\sharp \bigl(E \slash \mathord{\sim}_\nabla\bigr),
        \quad
	v_p \mapsto (p, [v_p])
    \end{equation*}
    we obtain a subbundle
    $F^\perp \coloneqq\Phi^{-1}\bigl(\pi^\sharp Q\bigr) \subseteq E$.
    This is clearly complementary to $F$ and also closed under
    $\nabla$ as $\Phi$ is compatible with covariant derivatives.
\end{proof}

A decomposition like that now allows us to construct a dual basis to
show regular projectivity of the triple of invariant sections:
\begin{proof}[Of \autoref{Prop:Set3Inj_projective_sections}]
    We choose a finite dual basis $(f_i,f^i)_{i \in I_\Null}$ of
    $\Secinfty_D(E_\Null)$ as in
    \autoref{Prop:Dual_Basis_Invariant_Sections} and a finite dual
    basis $(g_i,g^i)_{i \in I_\Null^\perp}$ of
    $\Secinfty_D(E_\Null^\perp)$.  Let $\Pi_\Null$
    (resp. $\Pi_\Null^\perp$) be the projections
    $\iota^\sharp E_\Wobs \to E_\Null$ (resp.
    $\iota^\sharp E_\Wobs \to E_\Null^\perp$).  We now include the
    $f_i,g_i$ into $\Secinfty(\iota^\sharp E_\Wobs)$ and pre-compose
    the $f^i,g^i$ with $\Pi_\Null$ (resp. $\Pi_\Null^\perp$), i.e.  we
    have $f^i \in \Secinfty((\iota^\sharp E_\Wobs)^*)$ with
    $f^i(v_q) = f^i(\Pi_\Null v_q)$ and
    $g^i(v_q) = g^i(\Pi_\Null^\perp v_p)$ for
    $v_q \in \iota^\sharp E_\Wobs$.  This way, we obtain a dual basis
    of $\Secinfty_D(\iota^\sharp E_\Wobs)$: Indeed, the projections
    fulfill
    $\nabla_X (\Pi_\Null \circ s) = \Pi_\Null \circ \nabla_X s$ and
    $\nabla_X (\Pi_\Null^\perp \circ s) = \Pi_\Null^\perp \circ
    \nabla_X s$ because
    \begin{equation*}
	\Pi_\Null \circ \nabla_X s
	=
        \Pi_\Null \circ \bigl( \nabla_X (\Pi_\Null \circ s)
	+
        \nabla_X (\Pi_\Null^\perp \circ s) \bigr)
	=
        \nabla_X( \Pi_\Null \circ s)
    \end{equation*}
    for $s \in \Secinfty(\iota^\sharp E_\Wobs)$.  Therefore, for
    $s \in \Secinfty_D(\iota^\sharp E_\Wobs)$ we have
    $\nabla_X \bigl(\Pi_\Null \circ s\bigr) = \Pi_\Null \circ \nabla_X
    s = 0$.  So the precomposed $f^i,g^i$ (which we still denote by
    the same symbol) map parallel sections in
    $\Secinfty_D(\iota^\sharp E_\Wobs)$ to functions in
    $\Cinfty_D(C)$.  Moreover, for $v_p \in \iota^\sharp E_\Wobs$ we
    have
    \begin{equation*}
	\sum_{i \in I_\Null} f_i(p)f^i(v_p)
        +
        \sum_{i \in I_\Null^\perp} g_i(p)g^i(v_p)
	=
        \sum_{i \in I_\Null} f_i(p)f^i(\Pi_\Null v_p)
	+
        \sum_{i \in I_\Null^\perp} g_i(p)g^i(\Pi_\Null^\perp v_p)
	=
        \Pi_\Null v_p + \Pi_\Null^\perp v_p
	=
        v_p .
    \end{equation*}
    Now we extend the $f_i,f^i,g_i,g^i$ to a dual basis of
    $\lbrace s \in \Secinfty(E_\Wobs\at{U}) \mid \nabla_X \iota^\sharp
    s = 0 \ \forall X \in \Secinfty(D) \rbrace$ for an open
    neighborhood $U \subseteq M$ of $C$: To do so, let
    $\Theta \colon E_\Wobs \at{U} \to t^\sharp (\iota^\sharp E_\Wobs)$
    be a vector bundle isomorphism as in
    \autoref{Lemma:Vector_Bundle_Isomorphism_Theta}.  We define
    \begin{equation*}
	\tilde{f}_i(p) \coloneqq \Theta^{-1}(p, f_i(t(p)))
    \end{equation*}
    for $p \in U$ which leads to $\iota^\sharp \tilde{f}_i = f_i$ and
    \begin{equation*}
	\tilde{f}^i_p(v_p) \coloneqq f^i_{t(p)}(t^\sharp(\Theta(v_p)))
    \end{equation*}
    for $v_p \in E_\Wobs\at{U}$ which leads to
    $\iota^\sharp \tilde{f}^i = f^i$.  We then have
    \begin{align*}
        \sum_{i \in I_\Null} \tilde{f}_i(p)
        & \tilde{f}^i(v_p)
        + \sum_{i \in I_\Null^\perp} \tilde{g}_i(p)\tilde{g}^i(v_p)
        \\
        &=
        \sum_{i \in I_\Null}
        \Theta^{-1}(p, f_i(t(p))) f^i_{t(p)}(t^\sharp(\Theta(v_p)))
        +
        \sum_{i \in I_\Null^\perp}
        \Theta^{-1}(p, g_i(t(p))) g^i_{t(p)}(t^\sharp(\Theta(v_p)))
        \\
        &=
        \Theta^{-1}\Bigl(p,
	\sum_{i \in I_\Null}
        f_i(t(p)) f^i_{t(p)}(t^\sharp(\Theta(v_p)))
	\sum_{i \in I_\Null^\perp}
        g_i(t(p)) g^i_{t(p)}(t^\sharp(\Theta(v_p)))
	\Bigr)
        \\
        &=
        \Theta^{-1}\Bigl(p, t^\sharp(\Theta(v_p)) \Bigr)
        \\
        &= v_p
    \end{align*}
    for $v_p \in E\at{U}$. Again, we denote the extensions by
    $f_i,f^i,g_i,g^i$.  Next, we choose a dual basis
    $(h_i,h^i)_{i \in J}$ of $\Secinfty(E_\Wobs\at{M \setminus C})$.
    Such a dual basis exists by \autoref{Theorem:Serre_Swan}.  Now we
    choose a quadratic partition of unity
    $\chi_1, \chi_2 \in \Cinfty(M, [0,1])$ with
    $\chi_1^2 + \chi_2^2 = 1$ and $\supp \chi_1 \subseteq U$ and
    $\supp \chi_2 \subseteq M \setminus C$. Existence is guaranteed by
    the closedness of $C$. Multiplying the $f_i,f^i,g_i,g^i$ with
    $\chi_1$ and the $h_i,h^i$ with $\chi_2$ yields sections
    $\tilde{f}_i,\tilde{f}^i,\tilde{g}_i,\tilde{g}^i,\tilde{h}_i,\tilde{h}^i$
    defined on the whole of $M$. Moreover, the
    $\tilde{h}_i, \tilde{h}^i$ vanish on $C$ and the
    $\tilde{f}_i,\tilde{f}^i,\tilde{g}_i,\tilde{g}^i$ vanish on
    $M \setminus U$. This also shows that the
    $\tilde{h}_i, \tilde{h}^i$ coincide with $h_i,h^i$ on
    $M \setminus U$ and that the
    $\tilde{f}_i,\tilde{f}^i,\tilde{g}_i,\tilde{g}^i$ coincide with
    the $f_i,f^i,g_i,g^i$ on $C$.  All those
    $\tilde{f}_i,\tilde{f}^i,\tilde{g}_i,\tilde{g}^i,\tilde{h}_i,\tilde{h}^i$
    together yield a dual basis of
    $\lbrace s \in \Secinfty(E_\Wobs) \mid \nabla_X \iota^\sharp s = 0
    \ \forall X \in \Secinfty(D) \rbrace$ as for $v_p \in E_\Wobs$
    with $p \in M \setminus U$ we have
    \begin{equation*}
        v_p
        =
        \sum_{i \in J} h_i(p) h^i(v_p)
        =
        \sum_{i \in J} \tilde{h}_i(p) \tilde{h}^i(v_p)
        =
        \sum_{i \in I_\Null} \tilde{f}_i(p) \tilde{f}^i(v_p)
        +
        \sum_{i \in I_\Null^\perp} \tilde{g}_i(p)\tilde{g}^i(v_p)
        +
        \sum_{i \in J} \tilde{h}_i(p) \tilde{h}^i(v_p),
    \end{equation*}
    if $p \in C$ we have
    \begin{equation*}
        v_p
        =
        \sum_{i \in I_\Null} f_i(p)f^i(v_p)
        +
        \sum_{i \in I_\Null^\perp} g_i(p) g^i(v_p)
        =
        \sum_{i \in I_\Null} \tilde{f}_i(p) \tilde{f}^i(v_p)
        +
        \sum_{i \in I_\Null^\perp} \tilde{g}_i(p)\tilde{g}^i(v_p)
        +
        \sum_{i \in J} \tilde{h}_i(p) \tilde{h}^i(v_p),
    \end{equation*}
    and if $p \in U \setminus C$ we have
    \begin{align*}
	v_p
        &=
        \chi_1^2(p) v_p + \chi_2^2(p) v_p
        \\
	&=
        \chi_1^2(p) \biggl(\sum_{i \in I_\Null} f_i(p)f^i(v_p)
	+ \sum_{i \in I_\Null^\perp} g_i(p) g^i(v_p) \biggr)
	+ \chi_2^2(p) \sum_{i \in J} h_i(p) h^i(v_p)
        \\
        &=
        \sum_{i \in I_\Null} \tilde{f}_i(p) \tilde{f}^i(v_p)
	+ \sum_{i \in I_\Null^\perp} \tilde{g}_i(p)\tilde{g}^i(v_p)
	+ \sum_{i \in J} \tilde{h}_i(p) \tilde{h}^i(v_p).
    \end{align*}
    For simplicity, we again denote the resulting
    $\tilde{f}_i,\tilde{f}^i,\tilde{g}_i,
    \tilde{g}^i,\tilde{h}_i,\tilde{h}^i$ by $f_i,f^i,g_i,g^i,h_i,h^i$.

    It remains to extend it to a dual basis of $\Secinfty(E_\Total)$.
    For this, we choose a subbundle $E_\Wobs^\perp$ complementary to
    $E_\Wobs$, i.e. $E_\Total = E_\Wobs \oplus E_\Wobs^\perp$.  We
    denote the projection onto $E_\Wobs$ by $\Pi_\Wobs$ and the
    projection onto $E_\Wobs^\perp$ by $\Pi_\Wobs^\perp$.  We choose a
    dual basis $(e_k,e^k)_{k \in K}$ of $\Secinfty(E_\Wobs^\perp)$.
    We can embed the $e_k$ into $\Secinfty(E_\Total)$ and turn the
    elements $e^k$ into elements of $\Secinfty(E_\Total^*)$ by
    precomposing with $\Pi_\Wobs^\perp$.  For
    $i \in I_\Null \sqcup I_\Null^\perp \sqcup J$ and
    $d \in \lbrace f,g,h \rbrace$ we set
    \begin{equation*}
        e_i \coloneqq d_i
        \quad
        \textrm{and}
        \quad
        e^i \coloneqq d^i \circ \Pi_\Wobs .
    \end{equation*}
    This gives us elements $e_i \in \Secinfty(E_\Total)$ and
    $e^i \in \Secinfty(E_\Total^*)$.

    Now we define the generator $N \in \SetTripleInj$ by
    $N_\Total \coloneqq I_\Null \sqcup I_\Null^\perp \sqcup J \sqcup
    K$, $N_\Wobs \coloneqq I_\Null \sqcup I_\Null^\perp \sqcup J$ and
    $N_\Null \coloneqq I_\Null$. As
    $\iota_N \colon N_\Wobs \to N_\Total$ we take the obvious
    inclusion.  So, our candidate for a $\SetTripleInj$-dual basis of
    $\Secinfty(\VBQuad{E})$ consists of the $f,g,h$ in the
    $\Wobs$-component, in the $\Null$-component we have the $f_i,f^i$
    and in the $\Total$-component we have the $e_i,e^i$ and the
    $e_k,e^k$. In the following, we check that this indeed fulfills
    the required properties: First of all, for $v_p \in E_\Wobs$ we
    have already seen that
    \begin{equation*}
	\sum_{i \in I_\Null \sqcup I_\Null^\perp \sqcup J}
	\bigl(d_i(p) d^i(v_p) \bigr) = v_p .
    \end{equation*}
    Moreover, on $C$ we have $\nabla_X d_i=0$ and $\nabla_X d^i=0$.
    For $i \in I_\Null$ we have
    $\iota^\sharp f_i \in \Secinfty_D(E_\Null)$ by definition of
    $f_i$.  This shows that for $i \in I_\Null$ we have
    $d_i \in \Secinfty(\VBQuad{E})_\Null$.

    For $i \in J$ the $h_i,h^i$ vanish on $C$ and therefore the
    condition $h^i(s) \in \algebra{J}_C$ for
    $s \in \Secinfty(\VBQuad{E})_\Null$ is trivially fulfilled.  For
    $i \in I_\Null^\perp$ it is fulfilled for the $g^i$ because
    $g^i(s)\at{C} = g^i\bigl(\Pi_\Null^\perp \circ s\bigr)\at{C} =
    g^i(0)\at{C} = 0$ for $s \in \Secinfty(E_\Null)$.  This shows that
    for $i \in I_\Null^\perp \sqcup J$ we have
    $d^i(s) \in \algebra{J}_C$ for
    $s \in \Secinfty(\VBQuad{E})_\Null$.

    The $e_i,e^i$ for $i \in N_\Total$ are a dual basis of
    $\Secinfty(E_\Total)$ because
    \begin{align*}
	\sum_{i \in I_\Null \sqcup I_\Null^\perp \sqcup J} e_i(p) e^i(v_p)
	+ \sum_{i \in K} e_i(p) e^i(v_p)
	&= \sum_{i \in I_\Null \sqcup I_\Null^\perp \sqcup J} d_i(p)
	d^i\bigl( \Pi_\Wobs(v_p)\bigr)
	+ \Pi_\Wobs^\perp(v_p)\\
	&= \Pi_\Wobs(v_p) + \Pi_\Wobs^\perp(v_p)\\
	&= v_p
    \end{align*}
    for all $v_p \in E_\Total$.  For $i \in N_\Wobs$ we have
    $e_{\iota_N(i)} = \iota_{\Secinfty(\VBQuad{E})}(d_i)$.  For
    $i \in N_\Total \setminus \iota_N(N_\Wobs)$, so $i \in K$, we have
    $e^i \circ \iota_{\Secinfty(\VBQuad{E})} = 0$ as
    $\Pi_\Wobs^\perp \at{E_\Wobs} = 0$.  For
    $i \in I_\Null \sqcup I_\Null^\perp \sqcup J$ we have
    $e^{\iota_N(i)} \circ \iota_{\Secinfty(\VBQuad{E})} = d^i$ as
    $\Pi_\Wobs \at{E_\Wobs} = \id_{E_\Wobs}$.

    This shows that we have finally found a $\SetTripleInj$-dual basis
    of $\Secinfty(\VBQuad{E})$. Therefore, $\Secinfty(\VBQuad{E})$ is
    a projective coisotropic $\Cinfty(M,C,D)$-module with finite
    generator $N \in \SetTripleInj$.
\end{proof}

Up to now, triples of invariant sections over triples of vector
bundles as in \autoref{Prop:Set3Inj_projective_sections} are merely an
example of regular projective coisotropic $\Cinfty(M,C,D)$-modules.
In the following, we want to see that essentially all regular
projective coisotropic $\Cinfty(M,C,D)$-modules are of this form. As a
first step, we formulate the following theorem, which does \emph{not}
need to assume any regularity of the integrable distribution $D$:
\begin{theorem}
    \label{Prop:Set3Inj_Proj_Modules_As_Sections}%
    Let $\module{P}$ be a finitely generated regular projective
    coisotropic module over the coisotropic algebra $\Cinfty(M,C,D)$.
    Assume that $M$ is connected and $C \slash D$ is connected with
    respect to the final topology induced by the projection
    $C \to C \slash D$.  Then there exists a vector bundle
    $E_\Total \to M$, vector subbundles $E_\Wobs \subseteq E_\Total$
    and $E_\Null \subseteq \iota^\sharp E_\Wobs$ and a holonomy-free
    $D$-connection $\nabla$ on $\iota^\sharp E_\Wobs$ such that
    $E_\Null$ is closed under $\nabla$ and the coisotropic module
    $\module{P}$ is isomorphic to
    $\Secinfty(E_\Total, E_\Wobs, E_\Null, \nabla)$.

    Moreover, we can choose $(E_\Total, E_\Wobs, E_\Null, \nabla)$ in
    such a way that there exists a decomposition
    $\iota^\sharp E_\Wobs = E_\Null \oplus E_\Null^\perp$ and
    $\nabla = \nabla^\Null \oplus \nabla^\perp$ where $\nabla^\Null$
    and $\nabla^\perp$ are holonomy-free $D$-connections on the
    subbundles $E_\Null$ and $E_\Null^\perp$, respectively.
\end{theorem}
\begin{proof}
    Let $\module{P} \simeq \image e$ for some
    $e = e^2 \colon \algebra{A}^{(N)} \to \algebra{A}^{(N)}$ with
    $N \in \SetTripleInj$ and $\algebra{A} = \Cinfty(M,C,D)$.  Then
    using the identification of module homomorphisms of free modules
    with matrices we find
    $\hat{e}_\Total \in M_{N_\Total}(\Cinfty(M))$ and
    $\hat{e}_\Wobs \in M_{N_\Wobs}(\Cinfty_D(M))$ with
    \begin{equation*}
	\Biggl(e_\Total
        \begin{psmallmatrix}
            f_1 \\ \vdots \\ f_{\abs{N_\Total}}
        \end{psmallmatrix}\Biggr)_i
        =
        \sum_{j \in N_\Total} (\hat{e}_\Total)_{ij} f_j
    \end{equation*}
    and similarly for the $\Wobs$-component. We define the vector
    bundle morphisms
    \begin{align*}
        \phi_\Total \colon
        M \times \field{R}^{N_\Total} \to M \times \field{R}^{N_\Total},
        \quad
        (p,v)
        &\mapsto
        (p, \hat{e}_\Total(p) v)
        \\
        \phi_\Wobs \colon
        M \times \field{R}^{N_\Wobs} \to M \times \field{R}^{N_\Wobs},
        \quad
        (p,v)
        &\mapsto
        (p, \hat{e}_\Wobs(p) v).
    \end{align*}
    Because $\hat{e}_\Total(p)^2 = \hat{e}_\Total(p)$ we have
    $\tr (\hat{e}_\Total(p)) = \dim \image \hat{e}_\Total(p) = \dim
    \ker (\Unit - \hat{e}_\Total(p)) \in \field{N}_0$.  As $M$ is
    assumed to be connected, the rank of $\phi_\Total$ (and similarly
    of $\phi_\Wobs$) is constant and their images define vector
    subbundles
    $E_\Total \coloneqq \image \phi_\Total \subseteq M \times
    \field{R}^{N_\Total}$ and
    $E_\Wobs \coloneqq \image \phi_\Wobs \subseteq M \times
    \field{R}^{N_\Wobs}$.  The map
    \begin{equation*}
        \psi \colon E_\Wobs \to E_\Total,
        \quad
        (p,v)
        \mapsto
        \left(
            p, \sum\nolimits_{i \in N_\Wobs} c_{\iota_N(i)} v^i
        \right)
    \end{equation*}
    with the standard basis $(c_j)$ of $\field{R}^{N_\Total}$ is an
    injective vector bundle morphism over the identity as
    $e_\Total \circ \iota_{\algebra{A}^{(N)}} =
    \iota_{\algebra{A}^{(N)}} \circ e_\Wobs$ and $\iota_N$ is assumed
    to be injective. Therefore, we can consider
    $E_\Wobs \subseteq E_\Total$ as a subbundle.

    Now we define a subbundle $E_\Null$ of $\iota^\sharp E_\Wobs$
    using the restriction of $\hat{e}_\Wobs$ to the subspace
    $V_\Null \subseteq \field{R}^{N_\Wobs}$ spanned by the standard
    basis $b_i$ for $i \in N_\Null$.  Indeed, we have
    \begin{equation*}
	\hat{e}_\Wobs(p)\at{V_\Null} \colon V_\Null \to V_\Null
    \end{equation*}
    for $p \in C$ as
    $e_\Wobs(\algebra{A}^{(N)}_\Null) \subseteq
    \algebra{A}^{(N)}_\Null$ and therefore
    $\bigl(\hat{e}_\Wobs(p)(v)\bigr)^i = 0$ for
    $i \in N_\Wobs \setminus N_\Null$.  Now we can proceed as before
    and define a vector bundle morphism
    \begin{equation*}
	\phi_\Null \colon C \times V_\Null \to C \times V_\Null,
	\quad
        (p,v) \mapsto (p, \hat{e}_\Wobs(p)(v))
    \end{equation*}
    over the identity. As the entries of $(\hat{e}_\Wobs)_{ij}\at{C}$
    are in $\Cinfty_D(C)$ we have a well-defined and continuous map
    \begin{equation*}
	C \slash D \ni [p]
        \mapsto
        \hat{e}_\Wobs(p)\at{V_\Null} \in M_{N_\Null}(\field{R}) .
    \end{equation*}
    Using $\hat{e}_\Wobs(p)^2 = \hat{e}_\Wobs(p)$ and the fact that
    $C \slash D$ is assumed to be connected with respect to the final
    topology we see analogously that $\phi_\Null$ has constant rank.
    This way we obtain a subbundle
    $E_\Null \coloneqq \image \phi_\Null \subseteq \iota^\sharp
    E_\Wobs$.  Choosing some complementary subspace $V_\Null^\perp$ to
    $V_\Null \subseteq \field{R}^{N_\Wobs}$ we obtain another
    subbundle
    $E_\Null^\perp \coloneqq \iota^\sharp E_\Wobs \cap \bigl( C \times
    V_\Null^\perp \bigr)$ of $\iota^\sharp E_\Wobs$. Indeed,
    $E_\Null^\perp$ is of constant rank because
    $E_\Null^\perp \oplus E_\Null = \iota^\sharp E_\Wobs$ and it is
    the intersection of two vector bundles over the manifold $C$.

    Having defined the subbundles, we now look at the covariant
    derivatives: On $\iota^\sharp E_\Wobs$ we can define a covariant
    derivative
    \begin{equation*}
        \nabla_X s
        \coloneqq
        \sum_{i, j \in N_\Wobs} b_i (\hat{e}_\Wobs(p))_{ij} \Lie_X(s^j),
    \end{equation*}
    where $X \in \Secinfty(TC)$,
    $s \in \Secinfty(\iota^\sharp E_\Wobs)$ and $\Lie_X$ denotes the
    Lie derivative of functions.  Because of
    $(\hat{e}_\Wobs)_{ij}\at{C} \in \Cinfty_D(C)$ and
    $s^i = \sum_{j \in N_\Wobs} (\hat{e}_\Wobs(p))_{ij} s^j$ for
    $s \in \Secinfty(\iota^\sharp E_\Wobs)$ we have
    \begin{equation*}
        \nabla_X s = \sum_{i \in N_\Wobs} b_i \Lie_X(s^i)
    \end{equation*}
    for $X \in \Secinfty(D)$ which makes it clear that the parallel
    transport is path-independent along paths inside the leaves of
    $D$.  It remains to show that the subbundles $E_\Null$ and
    $E_\Null^\perp$ are closed under $\nabla\at{D}$.  In fact, for any
    subspace $W \subseteq \field{R}^{N_\Wobs}$ it holds that
    $\Secinfty(\iota^\sharp E_\Wobs \cap (C \times W))$ is closed under
    differentiating in the direction of $D$: To see this, let
    $\Pi^\perp \colon \field{R}^{N_\Wobs} \to W^\perp$ be some linear
    projection onto a complementary subspace $W^\perp$ of $W$, so
    $\ker \Pi^\perp = W$.  Then for
    $s \in \Secinfty(\iota^\sharp E_\Wobs \cap (C \times W))$ and
    $v_p \in D$ we have
    \begin{equation*}
        \Pi^\perp \bigl(\nabla_{v_p} s \bigr)
        =
        \sum_{i \in N_\Wobs} \Pi^\perp(b_i) v_p(s^i)
        =
        \sum_{i, j \in N_\Wobs}
        b_j \Pi^\perp_{ij} v_p(s^i)
        =
        \sum_{i, j \in N_\Wobs}
        b_j v_p\bigl( \Pi^\perp_{ij} s^i \bigr)
        =
        \sum_{j \in N_\Wobs}
        b_j v_p\bigl( \bigl(\Pi^\perp (s)\bigr)^j \bigr)
        =
        0
    \end{equation*}
    for the \emph{constant} matrix coefficients $\Pi^\perp_{ij}$ of
    $\Pi^\perp$.  Therefore,
    $\nabla_{v_p} s \in \iota^\sharp E_\Wobs \cap (C \times W)$.  This
    shows that $E_\Null$ and $E_\Null^\perp$ are closed under
    $\nabla$.

    It remains to show that for
    $\VBQuad{E} \coloneqq (E_\Total, E_\Wobs,E_\Null, \nabla)$ the
    sections $\Secinfty(\VBQuad{E})$ are isomorphic to the module
    $\module{P}$ we started with.  For this, we show that
    $\Secinfty(\VBQuad{E}) \simeq \image e$.  We define
    $\Xi \colon \image(e) \to \Secinfty(\VBQuad{E})$ by
    \begin{align*}
        \Xi_\Total \colon \image(e_\Total) \ni
        \begin{psmallmatrix}
            \hat{s}^1 \\ \vdots \\ \hat{s}^{\abs{N_\Total}}
        \end{psmallmatrix}
        &\mapsto
        \bigl( p \mapsto (p, \hat{s}(p)) \bigr)
        \in \Secinfty(\VBQuad{E})_\Total
        \\
        \shortintertext{and}
        \Xi_\Wobs \colon \image(e_\Wobs) \ni
        \begin{psmallmatrix}
            \hat{r}^1 \\ \vdots \\ \hat{r}^{\abs{N_\Wobs}}
        \end{psmallmatrix}
        &\mapsto
        \bigl( p \mapsto (p, \hat{r}(p)) \bigr)
        \in \Secinfty(\VBQuad{E})_\Wobs .
    \end{align*}
    Here the components fulfill $\hat{s}^i \in \Cinfty(M)$ and
    $\hat{r}^i \in \Cinfty_D(M)$.  The map $\Xi_\Total$ indeed maps
    into $\Secinfty(E_\Total)$ because
    \begin{equation*}
        \hat{s} \in \image(e_\Total)
        \iff
        e_\Total(\hat{s}) = \hat{s}
        \iff
        \hat{e}_\Total(p) (\hat{s}(p))
        =
        \hat{s}(p) \; \forall p \in M ,
    \end{equation*}
    so $(p, \hat{s}(p)) \in (E_\Total)_p$ for all $p \in M$.
    Moreover, the second component of a section
    $s \in \Secinfty(E_\Total)$ is just an element $\hat{s}$ of
    $\Cinfty(M)^{N_\Total}$ which fulfills
    $\hat{e}_\Total(p) (\hat{s}(p)) = \hat{s}(p)$ for all $p \in M$.
    Therefore $\Xi_\Total$ is surjective.  Injectivity is clear.
    Next, $\Xi_\Wobs$ maps into $\Secinfty(\VBQuad{E})_\Wobs$
    because it maps into $\Secinfty(E_\Wobs)$ for the same reasons as
    above (the components now being in
    $\Cinfty_D(M) \subseteq \Cinfty(M)$) and because
    \begin{equation*}
        \bigl(\nabla_{v_p} (\iota^\sharp r)\bigr)^i
        =
        \sum_{j \in N_\Wobs} v_p(\hat{r}^j \circ \iota)
        =
        0
    \end{equation*}
    for $r \coloneqq \Xi_\Wobs(\hat{r})$ and $v_p \in D$ as
    $\hat{r}^j \circ \iota \in \Cinfty_D(C)$.  For surjectivity of
    $\Xi_\Wobs$ note that $\bigl(\nabla_{v_p} s \bigr)^i = v_p(s^i)$
    for $v_p \in D$, $s \in \Secinfty(\iota^\sharp E_\Wobs)$.  Then it
    is clear that the second component of an element in
    $\Secinfty(\VBQuad{E})_\Wobs$ is just an element
    $\hat{r} \in \Cinfty_D(M)^{N_\Wobs}$ which fulfills
    $\hat{e}_\Wobs(p) (\hat{r}(p)) = \hat{r}(p)$ for all $p \in M$.
    Therefore, $\Xi_\Wobs$ is surjective.  Clearly, $\Xi_\Wobs$ is
    injective.  We check the condition
    $\Xi_\Total \circ \iota_{\image(e)} =
    \iota_{\Secinfty(\VBQuad{E})} \circ \Xi_\Wobs$: Let
    $\hat{r} \in \image(e_\Wobs)$.  Then
    \begin{equation*}
        \Xi_\Total(\iota_{\image(e)}(\hat{r}))
        =
        \Xi_\Total
        \Bigl(\sum_{i \in N_\Wobs} c_{\iota_N(i)} \hat{r}^i\Bigr)
        =
        \Bigl(
        p
        \mapsto
        \bigl(
        p, \bigl(
        \sum_{i \in N_\Wobs} c_{\iota_N(i)} \hat{r}^i(p)
        \bigr)
        \bigr)
        \Bigr)
    \end{equation*}
    and
    \begin{align*}
        \iota_{\Secinfty(\VBQuad{E})}(\Xi_\Wobs(\hat{r}))
        &=
        \iota_{\Secinfty(\VBQuad{E})}
        \bigl( p \mapsto (p, \hat{r}(p))\bigr)
        \\
        &=
        \bigl( p \mapsto (p, \psi(\hat{r}(p)))\bigr)
        \\
        &=
        \Bigl(
        p \mapsto
        \Bigl(
        p,
        \sum_{i \in N_\Wobs} c_{\iota_N(i)} \hat{r}^i(p)
        \Bigr)
        \Bigr) .
    \end{align*}
    Therefore $\Xi$ is compatible with the $\iota$'s.  We now have to
    look at the $\Null$-component: If $\hat{r} \in \image(e)_\Null$,
    then by definition $\hat{r}^i$ is zero on $C$ for
    $i \in N_\Wobs \setminus N_\Null$.  In fact for
    $\hat{r} \in \image(e)_\Wobs$ we have
    $\hat{r} \in \image(e)_\Null$ iff $\hat{r}^i\at{C} = 0$ for all
    $i \in N_\Wobs \setminus N_\Null$: If $\hat{r}^i\at{C} = 0$ for
    all $i \in N_\Wobs \setminus N_\Null$, then
    $\hat{r} = e_\Wobs(\hat{r}) \in
    \image\bigl(e_\Wobs\at{\algebra{A}^{(N)}_\Null})\bigr) =
    \image(e)_\Null$.  Conversely if
    $\hat{r} \in \image\bigl(e_\Wobs\at{\algebra{A}^{(N)}_\Null}\bigr)
    = \image(e)_\Null$, then $\hat{r} \in \algebra{A}^{(N)}_\Null$ as
    $e_\Wobs$ respects the $\Null$-components.  This shows
    $\hat{r}^i\at{C} = 0$ for all $i \in N_\Wobs \setminus N_\Null$.
    With this equivalence we see that $(p, \hat{r}(p)) \in F$ for
    $p \in C$ and $\hat{r} \in \image(e)_\Null$.  Therefore
    $\Xi_\Wobs(\image(e)_\Null) \subseteq
    \Secinfty(\VBQuad{E})_\Null$.  Conversely, for
    $r = (p \mapsto (p, \hat{r}(p))) \in \Secinfty(\VBQuad{E})_\Null$
    we first of all have $\hat{r} \in \image(e)_\Wobs$ but also
    $\hat{r}^i\at{C} = 0$ for $i \in N_\Wobs \setminus N_\Null$.  This
    shows that $\hat{r} \in \image(e)_\Null$ and therefore
    $\Xi_\Wobs \at{\image(e)_\Null} \colon \image(e)_\Null \to
    \Secinfty(\VBQuad{E})_\Null$ is surjective.  Therefore $\Xi$ is a
    morphism of coisotropic $\algebra{A}$-modules which has bijective
    $\Xi_\Total$ and $\Xi_\Wobs$ and $\Xi_\Wobs$ is surjective on the
    $\Null$-components.  This is equivalent to being an isomorphism of
    coisotropic modules.
\end{proof}

\begin{remark}
    Of course the existence of such a decomposition in
    \autoref{Prop:Set3Inj_Proj_Modules_As_Sections} always implies
    that $E_\Null$ is closed under $\nabla$ and that $\nabla$ is a
    holonomy-free $D$-connection.  However, without assumptions on the
    structure of $C\slash D$ the converse is not clear, see
    \autoref{Lemma:Complementary_Subbundle_Closed}.
\end{remark}

In the singular case, we are not guaranteed that taking triples of
invariant sections gives us a finitely generated projective
$\Cinfty(M,C,D)$-module.  Nevertheless, if for simplicity we consider
$M = C$, $E_\Total = E_\Wobs = E$ and
$E_\Null = C \times \lbrace 0 \rbrace$ (or $=E$),
\autoref{Ex:Singular_Projective_1D} and
\autoref{Ex:Singular_Projective_2D} give us some positive examples in
a singular setting.

In the regular case, \autoref{Prop:Set3Inj_Proj_Modules_As_Sections}
lets us formulate an equivalence of categories as in
\autoref{Theorem:Serre_Swan}: To do so, we define the following
category:
\begin{definition}[The category $\VectTriple(M,C,D)$]
    Let $\iota \colon C \to M$ be a submanifold of a manifold $M$ and
    $D \subseteq TC$ an integrable distribution.  Then the category
    $\VectTriple(M,C,D)$ has
    \begin{definitionlist}
    \item as objects quadruplets
        $(E_\Total, E_\Wobs, E_\Null, \nabla)$ consisting of a vector
        bundle $E_\Total \to M$, a vector subbundle
        $E_\Wobs \subseteq E_\Total$ and a vector subbundle
        $E_\Null \subseteq \iota^\sharp E_\Wobs$ together with a
        holonomy-free $D$-connection on $\iota^\sharp E_\Wobs$ such
        that $E_\Null$ is closed under $\nabla$.
    \item as morphisms vector bundle morphisms
        $\Phi \colon E_\Total \to F_\Total$ over the identity such
        that
        \begin{equation}
            \Phi(E_\Wobs) \subseteq F_\Wobs, \quad
            \iota^\sharp\Phi(E_\Null) \subseteq F_\Null
        \end{equation}
        and
        \begin{equation}
            \iota^\sharp \Phi \circ (\nabla^E_X s)
            = \nabla^F_X ((\iota^\sharp \Phi_\Wobs) \circ s)
        \end{equation}
        for all $s \in \Secinfty(\iota^\sharp E_\Wobs)$ and all
        $X \in \Secinfty(D)$.
    \end{definitionlist}
\end{definition}
\begin{example}
    \label{ex:ClassicalVBTriple}%
    Consider the special case $M=C$ and $D=0$, corresponding to a
    classical manifold.  Then every vector bundle $E$ over $M$ can be
    considered as an object in $\VectTriple(M,M,0)$ with identical
    $\Total$-and $\Wobs$-component and vanishing
    $\Null$-component. The $D$-connection $\nabla$ is trivial since
    $D=0$.
\end{example}

Note that---similarly to
\autoref{Lemma:Parallel_Sections_Parallel_Transport}---the
infinitesimal condition of preserving the covariant derivative can be
integrated to a condition preserving the parallel transport:
\begin{lemma}
    \label{Lemma:Morphism_Covariant_Parallel}%
    Let $D \subseteq TC$ be an integrable distribution.  Let
    $\Phi\colon E \to F$ be a vector bundle morphism over the identity
    $\id_C$.  Let $\nabla^E$ and $\nabla^F$ be $D$-connections on $E$
    (resp. $F$).  Then the following is equivalent:
    \begin{lemmalist}
    \item The vector bundle morphism $\Phi$ fulfills
        \begin{equation}
            \Phi \circ \bigl( \nabla^E_X s\bigr) = \nabla^F_X ( \Phi \circ s)
	\end{equation}
        for sections $X\in \Secinfty(D)$ and $s\in\Secinfty(E)$.
    \item For any smooth curve $\gamma \colon I \to C$ inside a leaf
        of $D$ one has for $a,b \in I$ that
        \begin{equation}
            \Phi \circ \Parallel^E_{\gamma, a \to b}
            =
            \Parallel^F_{\gamma, a \to b} \circ \Phi.
	\end{equation}
    \end{lemmalist}
\end{lemma}

In the category $\VectTriple(M,C,D)$ we can turn taking invariant
sections into a functor:
\begin{definition}[Taking sections in $\VectTriple(M,C,D)$]
    Let $\iota \colon C \to M$ be a submanifold of a manifold $M$ and
    $D \subseteq TC$ an integrable distribution.  Then we define the
    functor
    \begin{equation}
	\Secinfty \colon \VectTriple(M,C,D)
	\to \cCoisoModTriple{\Cinfty(M,C,D)}
    \end{equation}
    by
    \begin{definitionlist}
    \item mapping objects
        $\VBQuad{E} \coloneqq (E_\Total, E_\Wobs, E_\Null, \nabla^E)$
        to the triple $\Secinfty(\VBQuad{E})$ of invariant sections.
    \item mapping morphisms $\Phi \colon \VBQuad{E} \to \VBQuad{F}$
        between
        $\VBQuad{E} \coloneqq (E_\Total, E_\Wobs, E_\Null, \nabla^E)$
        and
        $\VBQuad{F} \coloneqq (F_\Total, F_\Wobs, F_\Null, \nabla^F)$
        to
        \begin{equation}
            \begin{split}
                \Secinfty(\Phi)_\Total \colon
                \Secinfty(\VBQuad{E})_\Total \ni s
                &\mapsto
                \Phi \circ s \in \Secinfty(\VBQuad{F})_\Total
                \quad
                \textrm{and}
                \\
                \Secinfty(\Phi)_\Wobs \colon
                \Secinfty(\VBQuad{E})_\Wobs \ni r
                &\mapsto
                \Phi\at{E_\Wobs} \circ r
                \in \Secinfty(\VBQuad{F})_\Wobs.
            \end{split}
        \end{equation}
    \end{definitionlist}
\end{definition}
It is easily checked that this indeed gives a functor into
$\cCoisoModTriple{\Cinfty(M,C,D)}$.  By
\autoref{Prop:Set3Inj_projective_sections} we know that if $C$ is
closed and $D \subseteq TC$ is simple, then
\begin{equation}
    \Secinfty \colon \VectTriple(M,C,D)
    \to \Proj(\Cinfty(M,C,D)),
\end{equation}
where on the right hand side we have the full subcategory of finitely
generated regular projective coisotropic $\Cinfty(M,C,D)$-modules.
Moreover, we have seen in
\autoref{Prop:Set3Inj_Proj_Modules_As_Sections} that if $M$ and
$C \slash D$ are connected, then
\begin{equation}
    \Secinfty \colon \VectTriple(M,C,D)
    \to \Proj(\Cinfty(M,C,D))
\end{equation}
is essentially surjective.  It remains to show fullness and
faithfulness in order to prove the following result:
\begin{theorem}[Coisotropic Serre-Swan Theorem]
    \label{Thm:Coisotropic_Serre_Swan}%
    Let $\iota \colon C \to M$ be a closed submanifold of a connected
    manifold $M$ and let $D \subseteq TC$ be simple with connected
    leaf space $C \slash D$.  Then the functor
    \begin{equation}
        \Secinfty \colon \VectTriple(M,C,D)
        \to \Proj(\Cinfty(M,C,D))
    \end{equation}
    yields an equivalence of categories.
\end{theorem}
\begin{proof}
    Let $\VBQuad{E}, \VBQuad{F} \in \VectTriple(M,C,D)$.  We claim
    that
    \begin{equation*}
        \Hom_{\VectTriple}(\VBQuad{E}, \VBQuad{F}) \ni \Phi
        \mapsto
        (\Secinfty(\Phi)_\Total, \Secinfty(\Phi)_\Wobs)
        \in
        \Hom_{\cCoisoModTriple{\Cinfty(M,C,D)}}
        (\Secinfty(\VBQuad{E}), \Secinfty(\VBQuad{F}))
    \end{equation*}
    is a bijection. Injectivity is clear because if
    $\Secinfty(\Phi)_\Total = \Secinfty(\Phi^\prime)_\Total$, then
    $\Phi \circ s = \Phi^\prime \circ s$ for all
    $s \in \Secinfty(E_\Total)$.

    For surjectivity let a morphism $(\chi_\Total, \chi_\Wobs)$ of the
    coisotropic $\Cinfty(M,C,D)$-modules $\Secinfty(\VBQuad{E})$ and
    $\Secinfty(\VBQuad{F})$ be given.  We know that as
    $\chi_\Total \in \Hom_{\Cinfty(M)}(\Secinfty(E_\Total),
    \Secinfty(F_\Total)) \simeq \Hom_{\Vect(M)}(E_\Total,F_\Total)$
    there exists a (unique) vector bundle morphism $\Phi$ over $\id_M$
    with $\Phi \circ s = \chi_\Total(s)$ for all
    $s \in \Secinfty(E_\Total)$.  For $v_p \in E_\Wobs$ the section
    $s$ defined by
    \begin{equation*}
        s(q) \coloneqq \sum_{i \in N_\Wobs}d_i(q) d^i(v_p)
    \end{equation*}
    with a dual basis as in \autoref{Prop:Set3Inj_projective_sections}
    is an element of $\Secinfty(\VBQuad{E})_\Wobs$ and fulfills
    $s(p) = v_p$. Therefore
    \begin{equation*}
        \Phi(v_p)
        =
        \chi_\Total(\iota_1(s))\at{p}
        =
        \iota_2(\chi_\Wobs(s))\at{p} \in F_\Wobs
    \end{equation*}
    as $\chi_\Wobs(s) \in \Secinfty(\VBQuad{F})_\Wobs$. Similarly, for
    $v_p \in E_\Null$ the section $s$ defined by
    \begin{equation*}
        s(q) \coloneqq \sum_{i \in I_\Null}f_i(q) f^i(v_p)
    \end{equation*}
    with the $f_i,f^i$ as in
    \autoref{Prop:Set3Inj_projective_sections} is an element of
    $\Secinfty(\VBQuad{E})_\Null$ and fulfills $s(p) = v_p$. Therefore
    \begin{equation*}
        \Phi(v_p) = \chi_\Total(\iota_1(s))\at{p} = \iota_2(\chi_\Wobs(s))\at{p} \in F_\Null
    \end{equation*}
    as $\chi_\Wobs(s) \in \Secinfty(\VBQuad{F})_\Null$. This shows
    that
    \begin{equation*}
        \Phi(E_\Wobs) \subseteq F_\Wobs
        \quad
        \textrm{and}
        \quad
        \iota^\sharp\Phi(E_\Null) \subseteq F_\Null .
    \end{equation*}
    We need to check that $\iota^\sharp \Phi_\Wobs$ preserves the
    covariant derivatives. For this let
    $v_p \in \iota^\sharp E_\Wobs$. Then as already used above there
    exists a section $s \in \Secinfty(\VBQuad{E})_\Wobs$ with
    $s(p) = v_p$ and for $\gamma$ being an arbitrary curve in a leaf
    of $D$ connecting $p,q \in C$ we have
    \begin{equation*}
        \Parallel^F_{\gamma, p \to q}(\Phi(v_p))
        =
        \Parallel^F_{\gamma, p \to q}(\chi_\Wobs(s)\at{p})
        =
        \chi_\Wobs(s)\at{q}
        =
        \Phi(s(q))
        =
        \Phi\bigl(\Parallel^E_{\gamma, p \to q}(s(p))\bigr)
        =
        \Phi\bigl(\Parallel^E_{\gamma, p \to q}(v_p)\bigr),
    \end{equation*}
    because of
    $\chi_\Wobs(s)\at{C} \in \Secinfty_D(\iota^\sharp F_\Wobs,
    \nabla^F)$ and
    \autoref{Lemma:Parallel_Sections_Parallel_Transport}.  By
    \autoref{Lemma:Morphism_Covariant_Parallel} this shows that
    $\iota^\sharp \Phi_\Wobs$ preserves the covariant derivatives.
    Therefore, $\Phi$ is a preimage of $(\chi_\Total, \chi_\Wobs)$ in
    $\Hom_{\VectTriple(M,C,D)}(\VBQuad{E}, \VBQuad{F})$.
\end{proof}
\begin{remark}
    The classical Serre-Swan Theorem can be understood as a special
    case of \autoref{Thm:Coisotropic_Serre_Swan} as follows: By
    \autoref{ex:ClassicalVBTriple} the category $\Vect(M)$ of vector
    bundles over a manifold $M$ can be embedded into the category
    $\VectTriple(M,M,0)$.  Similarly, the category $\Proj(\Cinfty(M))$
    of finitely generated projective modules over $\Cinfty(M)$ can be
    embedded into the category $\Proj(\Cinfty(M,M,0))$, see
    \autoref{ex:ClassicalProjModule}.  Then it is easy to see that the
    functor of taking sections preserves these subcategories,
    reproducing the classical Serre-Swan Theorem.
\end{remark}

%
% Reduction of Regular Projective Coisotropic $\Cinfty(M,C,D)$-Modules
%

\subsection{Reduction of Regular Projective Coisotropic $\Cinfty(M,C,D)$-Modules}

As a final step, we want to look at the reduction functor
$\red \colon \cCoisoModTriple{\algebra{A}} \to
\cModules{\algebra{A}_\red}$: Knowing that---under some regularity
assumptions---regular projective coisotropic $\Cinfty(M,C,D)$-modules
are essentially triples of vector bundles together with a
holonomy-free $D$-connection, we can of course ask whether the
quotient from \autoref{Prop:Quotient_Vector_Bundles} is in some sense
compatible with reduction of coisotropic modules, i.e.  taking the
quotient of the $\Wobs$- and the $\Null$-component.  As we will see in
the following, this is indeed the case.  However, we need yet another
quotient of vector bundles in order to cope with the sections taking
values in $E_\Null$: For a vector bundle $E \to C$ and a subbundle
$F \subseteq E$ we can consider the fiberwise quotient vector bundle
$\frac{E}{F} \to C$.  It has the property that the projection
$\Pi \colon E \to \frac{E}{F}$ is both a submersion and a vector
bundle morphism.  Moreover, if $\nabla$ is a $D$-connection on $E$
such that $F$ is closed under $\nabla$, then there exists a unique
$D$-connection $\widehat{\nabla}$ on $\frac{E}{F}$ which fulfills
\begin{equation}
    \label{Eq:Reduced_Nabla}
    \widehat{\nabla}_X\bigl( \Pi \circ s \bigr)
    =
    \Pi \circ \bigl( \nabla_X s \bigr)
\end{equation}
for all $X \in \Secinfty(D), s \in \Secinfty(E)$.  Note also that for
any complementary vector subbundle $F^\perp$ to $F$ the map
\begin{equation}
    \psi \colon F^\perp \to \tfrac{E}{F},
    \quad
    v_p \mapsto [v_p]
\end{equation}
is an isomorphism of vector bundles. Using this, it is straightforward
to prove the following lemma:
\begin{lemma}
    \label{Lemma:Section_Reduction_Complementary}%
    Let $\iota \colon C \to M$ be a closed submanifold of a manifold
    $M$ and $D \subseteq TC$ an integrable distribution.  Let
    $E_\Wobs \to M$ be a vector bundle and
    $E_\Null \subseteq \iota^\sharp E_\Wobs$ a subbundle. Assume that
    $\nabla$ is a $D$-connection on $\iota^\sharp E_\Wobs$, such that
    $E_\Null$ is closed under $\nabla$ and there exists a
    complementary subbundle $E_\Null^\perp$ which is also closed under
    $\nabla$.  Then for
    \begin{equation}
        \Secinfty_D(E_\Wobs, \nabla)
        \coloneqq
        \lbrace
        s \in \Secinfty(E_\Wobs)
        \mid
        \nabla_X \iota^\sharp s = 0
        \quad
        \forall X \in \Secinfty(D)
        \rbrace
    \end{equation}
    and
    \begin{equation}
        \Secinfty_D(E_\Wobs, \nabla)_\Null
        \coloneqq
        \lbrace
        s \in \Secinfty_D(E_\Wobs, \nabla)
        \mid
        \iota^\sharp s \in \Secinfty(E_\Null)
        \rbrace
    \end{equation}
    the map
    \begin{equation}
        \label{Eq:Isomorphism_Quotient_Sections}
	\Secinfty_D(E_\Wobs, \nabla)
        \slash
        \Secinfty_D(E_\Wobs, \nabla)_\Null
	\ni [s]
        \mapsto
        \Pi \circ \iota^\sharp s \in
	\Secinfty_D
        \Bigl(
        \tfrac{\iota^\sharp E_\Wobs}{E_\Null}, \widehat{\nabla}
        \Bigr)
    \end{equation}
    with
    $\Pi \colon \iota^\sharp E_\Wobs \to \frac{\iota^\sharp
      E_\Wobs}{E_\Null}$ and $\widehat{\nabla}$ as in
    \eqref{Eq:Reduced_Nabla} defines an isomorphism of
    $\Cinfty_D(C)$-modules.
\end{lemma}

In the case of a simple $D \subseteq TC$ we can thus prove that
reduction is compatible with taking sections:
\begin{theorem}[Reduction of triples of invariant sections]
    Let $\iota \colon C \to M$ be a closed submanifold and let
    $D \subseteq TC$ be simple.  Moreover, let
    $\VBQuad{E} \in \VectTriple(M,C,D)$.  Then
    \begin{equation}\label{Eq:Reduced_Sections_Triple}
	\bigl(\Secinfty(\VBQuad{E}) \bigr)_\red
	\simeq \Secinfty\Bigl( \tfrac{\iota^\sharp E_\Wobs}{E_\Null}
	\slash \mathord{\sim}_{\widehat{\nabla}} \Bigr)
    \end{equation}
    with $\widehat{\nabla}$ as in \eqref{Eq:Reduced_Nabla}. More
    precisely, the diagram
    \begin{equation}
        \begin{tikzcd}
            \VectTriple(M,C,D)
            \arrow{r}{\red}
            \arrow{d}[swap]{\Secinfty}
            & \Vect(C \slash D)
            \arrow{d}{\Secinfty}
            \\
            \Proj(\Cinfty(M,C,D))
            \arrow{r}{\red}
            & \Proj(\Cinfty(C \slash D))
        \end{tikzcd}
    \end{equation}
    commutes up to a natural isomorphism, where
    $\red \colon \VectTriple(M,C,D) \to \Vect(C \slash D)$ is the
    functor mapping objects $(E_\Total, E_\Wobs, E_\Null, \nabla)$ to
    $\frac{\iota^\sharp E_\Wobs}{E_\Null} \slash
    \mathord{\sim}_{\widehat{\nabla}}$ and morphisms to their obvious
    induced counterparts.
\end{theorem}
\begin{proof}
    By \autoref{Lemma:Complementary_Subbundle_Closed} there exists a
    complementary subbundle $E_\Null^\perp$, which is closed under
    $\nabla$.  Thus by \autoref{Lemma:Section_Reduction_Complementary}
    we have
    \begin{equation*}
        \bigl(\Secinfty(\VBQuad{E}) \bigr)_\red
        \simeq
        \Secinfty_D
        \Bigl(
        \tfrac{\iota^\sharp E_\Wobs}{E_\Null}, \widehat{\nabla}
        \Bigr).
    \end{equation*}
    Using the fact that
    $\Pi \colon \iota^\sharp E_\Wobs \to \frac{\iota^\sharp
      E_\Wobs}{E_\Null}$ is compatible with the covariant derivatives
    we see that by \autoref{Lemma:Morphism_Covariant_Parallel}
    $\widehat{\nabla}$ is a holonomy-free $D$-connection.  Thus
    \autoref{Lemma:Isomorphism_Invariant_Sections_Quotient_Sections}
    yields the first result.  To see that the diagram of functors
    commutes up to natural isomorphism first note that
    $\red \colon \VectTriple(M,C,D) \to \Vect(C \slash D)$ maps
    morphisms $\Phi \colon \VBQuad{E} \to \VBQuad{F}$ in
    $\VectTriple(M,C,D)$ to vector bundle morphisms $\phi$ over the
    identity $\id_{C \slash D}$ explicitly given by
    \begin{equation*}
        \phi\Bigl( \bigl[ [v_p]_{E_\Null} \bigr]_{\widehat{\nabla}^E} \Bigr)
        =
        \bigl[ [\Phi(v_p)]_{F_\Null} \bigr]_{\widehat{\nabla}^F}.
    \end{equation*}
    The map $\phi$ is well-defined by the assumptions on $\Phi$ and a
    vector bundle morphism because the projections
    $\iota^\sharp E_\Wobs \to \frac{\iota^\sharp E_\Wobs}{E_\Null}$,
    $\frac{\iota^\sharp E_\Wobs}{E_\Null} \to \frac{\iota^\sharp
      E_\Wobs}{E_\Null} \slash \mathord{\sim}_{\widehat{\nabla}^E}$,
    $\iota^\sharp F_\Wobs \to \frac{\iota^\sharp F_\Wobs}{F_\Null}$
    and
    $\frac{\iota^\sharp F_\Wobs}{F_\Null} \to \frac{\iota^\sharp
      F_\Wobs}{F_\Null} \slash \mathord{\sim}_{\widehat{\nabla}^F}$
    are submersions and vector bundle morphisms.  Then, unwinding the
    definitions of the functors, one sees that the diagram commutes up
    to the natural isomorphisms $\eta$ given by
    \begin{align*}
	\eta_{\VBQuad{E}} \colon
        \bigl(\Secinfty(\VBQuad{E}) \bigr)_\red
	&\to
        \Secinfty\bigl(
        \tfrac{\iota^\sharp E_\Wobs}{E_\Null}
        \slash
        \mathord{\sim}_{\widehat{\nabla}}
        \bigr),
        \\
	[s] \mapsto \Bigl( C \slash D \ni q
        &\mapsto
	\Bigl[
        \bigl[
        \iota^\sharp s(\pi^{-1}(q)
        \bigr]_{E_\Null}
        \Bigr]_{\widehat{\nabla}}
        \in
	\tfrac{\iota^\sharp E_\Wobs}{E_\Null}
        \slash
        \mathord{\sim}_{\widehat{\nabla}} \Bigr]
        \Bigr),
    \end{align*}
    where $\pi^{-1}(q) \in C$ is an arbitrary element of
    $\pi^{-1}(\lbrace q \rbrace)$: This $\eta_{\VBQuad{E}}$ is just
    the composition of \eqref{Eq:Isomorphism_Quotient_Sections} and
    \eqref{Eq:Isomorphism_Invariant_Sections_Inverse}.
\end{proof}

A result like this was to be expected: In
\autoref{prop:ProjectiveReduction} we already saw that the reduction
of a regular projective coisotropic $\Cinfty(M,C,D)$-module yields a
projective $\Cinfty_D(C) \simeq \Cinfty(C \slash D)$-module.
\autoref{Theorem:Serre_Swan} tells us that those are essentially given
by sections of vector bundles over $C \slash D$.

\begin{example}[Bott connection]
    Let $D \subseteq TC$ be an integrable distribution of constant
    rank.  The Bott connection with
    \begin{equation}
        \nabla \colon
        \Secinfty(D) \times \Secinfty\bigl(\tfrac{TC}{D}\bigr)
        \to
        \Secinfty\bigl(\tfrac{TC}{D}\bigr),
        \quad
        \nabla_X (\Pi \circ Y) = \Pi([X,Y])
    \end{equation}
    for $X \in \Secinfty(D), Y \in \Secinfty(TC)$ and the projection
    $\Pi \colon TC \to \frac{TC}{D}$ is flat in direction $D$,
    cf. \cite[Lemma~6.3]{bott:1972a}.  It is not hard to check that if
    $\pi \colon C \to C \slash D$ is a submersion,
    \begin{equation}
        \tfrac{TC}{D} \ni \Pi(v_p)
        \mapsto
        (p, T\pi(v_p))
        \in \pi^\sharp\bigl(T\bigl( C \slash D\bigr)\bigr)
    \end{equation}
    is a vector bundle isomorphism over the identity and
    \begin{equation}
        \tfrac{TC}{D} \!\bigm\slash \mathord{\sim}_\nabla
        \,\simeq
        T\bigl(C \slash D\bigr) .
    \end{equation}
    To interpret this in the context of coisotropic triples, first we
    need to come up with a $D$-connection $\widetilde{\nabla}$ on $TC$
    which induces $\nabla$ on $\frac{TC}{D}$ via
    \eqref{Eq:Reduced_Nabla}.  This can be done by choosing an
    arbitrary complement $D^\perp \subseteq TC$ to $D$ and using the
    isomorphism
    $\psi \colon D^\perp \ni v_p \mapsto [v_p] \in \frac{TC}{D}$ of
    vector bundles to carry $\nabla$ over to $D^\perp$, yielding some
    $\widetilde{\nabla}^{D^\perp}$.  Now we need to choose an
    arbitrary $D$-connection $\widetilde{\nabla}^D$ on the vector
    bundle $D$. Their direct sum
    $\widetilde{\nabla} \coloneqq \widetilde{\nabla}^D \oplus
    \widetilde{\nabla}^{D^\perp}$ then is a $D$-connection on $TC$
    inducing $\nabla$ on $\frac{TC}{D}$ via \eqref{Eq:Reduced_Nabla}.
    Hence, \autoref{Lemma:Section_Reduction_Complementary} implies
    \begin{equation}
        \Secinfty_D(TC, \widetilde{\nabla})
        \bigm\slash
        \Secinfty_D(D, \widetilde{\nabla})
        \simeq
        \Secinfty_D\bigl(\tfrac{TC}{D}, \nabla \bigr)
        \simeq
        \Secinfty\Bigl(\bigl( \tfrac{TC}{D}, \nabla\bigr)
        \bigm\slash
        \mathord{\sim}_\nabla\Bigr)
        \simeq
        \Secinfty\bigl(T(C\slash D)\bigr) .
    \end{equation}
    To incorporate this construction into the setting of coisotropic
    $\Cinfty(M,C,D)$-modules, we might need to restrict to a tubular
    neighborhood of $C$:
    \autoref{Lemma:Vector_Bundle_Isomorphism_Theta} allows to find a
    neighborhood $U \subseteq M$ of a closed submanifold
    $\iota \colon C \to M$ such that
    $TM\at{U} = TU \simeq t^\sharp(\iota^\sharp(TM))$ for a submersion
    $t \colon U \to C$ with $t \circ \iota = \id_C$. Therefore,
    \begin{equation}
        \widetilde{TC} \coloneqq t^\sharp(TC)
        \subseteq
        t^\sharp(\iota^\sharp(TM)) \simeq TU
    \end{equation}
    defines a subbundle of $TU$ which restricts to $TC$ on $C$.  We
    arrive at the coisotropic $\Cinfty(U,C,D)$-module
    $\Secinfty(TU, \widetilde{TC}, D, \widetilde{\nabla})$ with
    \begin{equation}
	\Secinfty\bigl(
        TU, \widetilde{TC}, D, \widetilde{\nabla}
        \bigr)_\red
	\simeq
        \Secinfty\bigl(T(C\slash D)\bigr)
    \end{equation}
    by \eqref{Eq:Reduced_Sections_Triple}.  If in addition
    $\widetilde{\nabla}$ is a holonomy-free $D$-connection also on
    $D$, (it is always holonomy-free on $D^\perp$), then
    \autoref{Prop:Set3Inj_projective_sections} tells us that
    $\Secinfty(TU, \widetilde{TC}, D, \widetilde{\nabla})$ is a
    finitely generated regular projective coisotropic
    $\Cinfty(U,C,D)$-module.
\end{example}

%
% The bibliographies, uncomment if you need to cite someone
% Make it smaller than other text
%

{
  \footnotesize
%  \printbibliography[heading=bibintoc]
%  \bibliographystyle{nchairx}
%
%\bibliography{dqarticle,dqbook,dqprocentry,dqproceeding,dqthesis,preprints,misc,notes,script,../include/dmw}

}

%
% this is the end of dmw
%

\end{document}